\documentclass[12pt]{amsart}
\usepackage{mathrsfs}

\usepackage{amssymb,amsfonts,amsthm,amsmath}
\usepackage{epsfig}
\usepackage[utf8]{inputenc}
\usepackage{mathrsfs}
\usepackage{lscape}
\usepackage{amssymb,amsmath,graphicx,color,textcomp, amsthm,bbm,bbold, enumerate,booktabs}
\usepackage{chngcntr}
\usepackage{apptools}
\AtAppendix{\counterwithin{theorem}{section}}
\definecolor{Red}{cmyk}{0,1,1,0}

\definecolor{verde}{cmyk}{1,0,1,0}

\definecolor{loka}{cmyk}{.5,0,1,.5}
\definecolor{azul}{cmyk}{1,1,0,0}

\numberwithin{equation}{section}

\newcommand{\be}{\begin{equation}}
\newcommand{\ee}{\end{equation}}

\newtheorem{definition}{Definition}

\newtheorem{theorem}{Theorem}
\newtheorem{remark}{Remark}
\newtheorem{lemma}{Lemma}

\begin{document}
\title{On the stability of a hyperbolic fractional partial differential equation}
\author{J. Vanterler da C. Sousa$^1$}
\address{$^1$ Department of Applied Mathematics, Institute of Mathematics,
 Statistics and Scientific Computation, University of Campinas --
UNICAMP, rua S\'ergio Buarque de Holanda 651,
13083--859, Campinas SP, Brazil\newline
e-mail: {\itshape \texttt{vanterlermatematico@hotmail.com, capelas@ime.unicamp.br }}}
\author{E. Capelas de Oliveira$^1$}

\begin{abstract} In this paper, by means of the Gronwall inequality, the $\psi$-Riemann-Liouville fractional partial integral and the $\psi$-Hilfer fractional partial derivative are introduced and some of its particular cases are recovered. Using these results, we investigate the Ulam-Hyers and Ulam-Hyers-Rassias stabilities of the solutions of a fractional partial differential equation of hyperbolic type in a Banach space	$(\mathbb{B}, \left\vert \cdot\right\vert)$, real or complex.

\end{abstract}
\maketitle

\section{Introduction} 

Fractional calculus is currently one of the subjects most studied by mathematicians and also by researchers in physics, chemistry, engineering, among others, due to its innumerable applications in modeling real phenomena. Indeed, by considering non-integer order derivatives, it has been possible, in several cases, to better match theoretical models to experimental data, allowing better predictions of the future dynamics of processes 
\cite{KDI,ELM,HER,KSTJ,SAMKO}. One problem in this field is the great number of possible different definitions of fractional operators, making the choice of the best operator for each particular system a crucial issue. One way to overcome this problem is to consider more general definitions, of which the usual ones can be considered particular cases \cite{KSTJ,SAMKO,OLMS}. In this sense, Sousa and Oliveira \cite{ZE1} recently introduced a fractional derivative of one variable, the so-called $\psi$-Hilfer fractional derivative, from which
it is possible to obtain a wide class of fractional derivatives already well established. Therefore, one of the objectives of this paper is to introduce the $\psi$-Riemann-Liouville fractional partial integral of a function with respect to another function and the $\psi$-Hilfer fractional partial derivative of $N$ variables.

On the other hand, the use of fractional derivatives in systems of differential equations has found several applications in mathematical models for population dynamics, erythrocyte sedimentation rate and others
\cite{KDI,JLEC1,sao,fra,almeida}. In addition, there has been a significant development in the study of ordinary differential equations involving derivatives of fractional order (see \cite{B8} and references therein). In
particular, many studies have focused on the study of the Ulam-Hyers stability of solutions of integrodifferential equations of fractional order \cite{est1,est2,est3,est4}. 

In the last years, Lungu et al., Brzdek et al. and Rus et al. published several papers on the study of the Ulam stability of the solution of partial differential equations of hyperbolic and pseudoparabolic types \cite{ulam,ulam1,ulam2,ulam3}. In some of these cases, the authors made use of the Gronwall inequality and studied the Ulam-Hyers stability in Banach spaces \cite{ulam,ulam1}. On the other hand, Abbas and Benchohra \cite{B1,B2,B4}, Elemad and Rezapour \cite{B3}, Abbas et al. \cite{B5}, among others \cite{B8,B9}, devoted themselves to studying the existence, uniqueness and the Ulam-Hyers stability of solutions of differential equations of fractional order. We can also mention the study of functional partial differential equations and the Darboux problem for partial differential equations of fractional order \cite{B6,B7}.

In this paper, we consider the following fractional partial differential equation of hyperbolic type: 
\begin{equation}\label{eq3}
\dfrac{\partial _{\beta ;\psi }^{2\alpha }u}{\partial _{\beta ;\psi }^{\alpha }x^{\alpha }\partial _{\beta ;\psi }^{\alpha }y^{\alpha }}\left( x,y\right)= f\left( x,y,u\left( x,y\right) ,\dfrac{\partial _{\beta ;\psi }^{\alpha }u}{\partial _{\beta ;\psi }^{\alpha }x^{\alpha }}\left( x,y\right)  ,\dfrac{\partial _{\beta ;\psi }^{\alpha }u}{\partial _{\beta ;\psi }^{\alpha}y^{\alpha }}\left( x,y\right) \right),
\end{equation}
where $\dfrac{\partial _{\beta ;\psi }^{2\alpha }u}{\partial _{\beta ;\psi }x^{\alpha }\partial_{\beta ;\psi }y^{\alpha }}\left( \cdot, \cdot\right)$ is the $\psi$-Hilfer fractional partial derivative with $\dfrac{1}{2}<\alpha \leq 1$, $0\leq \beta \leq 1$ and $\text{ }0\leq x<a,\text{ }0\leq y<b$ and $f\in
C\left( \left[ 0,a\right) \times \left[ 0,b\right) \times \mathbb{B}^{3},\mathbb{B}\right)$, where $(\mathbb{B}, \left\vert \cdot\right\vert)$ is a real or complex Banach space.

The main objective of this paper is to introduce the $\psi$-Riemann-Liouville fractional partial integral and the $\psi$-Hilfer fractional partial derivative of $N$ variables, in order to study the Ulam-Hyers and Ulam-Hyers-Rassias stabilities of the solutions of Eq.(\ref{eq3}) by means of the $\psi$-Hilfer fractional partial derivative and the Gronwall inequality.

The paper is organized as follows: In section 2, we present definitions and results considered important for the development of the article. In section 3, we introduce the $\psi$-Riemann-Liouville fractional integral of a function with respect to another function of $N$ variables and the $\psi$-Hilfer fractional partial derivative and study some of its particular cases, highlighting relevant points.  Moreover, using the $\psi$-Hilfer fractional partial derivative, we present new versions for the definitions of Ulam-Hyers and Ulam-Hyers-Rassias stabilities. In section 4, by means of Theorem \ref{ver1}, we study the Ulam-Hyers stability, while in section 5,
through Theorem \ref{ver2}, we discuss the generalized Ulam-Hyers-Rassias stability. Concluding remarks close the paper.

\section{Preliminaries} 

In this section we present the definitions of the $\psi$-Riemann-Liouville fractional integral and the $\psi$-Hilfer fractional derivative; we introduce the Gronwall inequality by means of a lemma and present a particular case of Eq.(\ref{eq1}). 

\begin{definition}\textnormal{\cite{ZE1}} Let $\left( a,b\right) $ $\left( -\infty \leq a<b\leq \infty \right) $ be a finite (or infinite) interval of the real line $\mathbb{R}$ and let $\alpha >0$. Also let	$\psi \left( t\right) $ be an increasing and positive monotone function on $\left( a,b\right]$ having a continuous derivative $\psi^{\prime }\left( t\right)$ on $\left( a,b\right)$.  We denote its first 	derivative as $\dfrac{d}{dt}\psi(t)=\psi'(t)$.  The left-sided fractional integral of a function $f$ with respect to a function $\psi
$ on $ \left[ a,b\right] $ is defined by 
\begin{equation}\label{eq1}
I_{a+}^{\alpha ;\psi }f\left( t\right) =\frac{1}{\Gamma \left( \alpha \right) }\int_{a}^{t}\psi ^{\prime }\left( s\right) \left( \psi \left( t\right) -\psi 
\left( s\right) \right) ^{\alpha -1}f\left( s\right) ds.
\end{equation}
The right-sided fractional integral is defined in an analogous form.
\end{definition}

\begin{definition}{\rm\cite{ZE1}} Let $n-1<\alpha <n$ with $n\in \mathbb{N}$; let $\ I=\left[ a,b\right] $ be an interval such that	$-\infty \leq a<b\leq \infty $ and let $f,\psi \in C^{n}\left[a,b\right] $ be two functions such that $\psi $ is increasing and $\psi ^{\prime }\left( t\right) \neq 0$ for all $t\in I$.
The left-sided $\psi$-Hilfer fractional derivative, denoted by $^{H}\mathbb{D}_{a+}^{\alpha ,\beta ;\psi }\left( \cdot \right) $, of a function $f$, of order $\alpha $ and type $0\leq \beta\leq 1,$ is defined by 
\begin{equation}\label{eq2}
^{H}\mathbb{D}_{a+}^{\alpha ,\beta ;\psi }f\left( t\right) =I_{a+}^{\beta \left( n-\alpha \right) ;\psi }\left( \frac{1}{\psi ^{\prime }\left( t\right) }\frac{d}{dt}\right) ^{n}I_{a+}^{\left( 1-\beta \right) \left(
n-\alpha \right) ;\psi }f\left( t\right) .
\end{equation}
The right-sided $\psi$-Hilfer fractional derivative is defined in an analogous form.
\end{definition}

The following Gronwall lemma is an important tool in proving the main results of this paper.

\begin{lemma}\textnormal{\cite{gronwall}} \label{l1} We assume that: 
\begin{enumerate}

\item  $u,v,h\in C\left( \mathbb{R}_{+}^{n},\mathbb{R}_{+}\right)$;
	
\item $\psi \left( t\right)$ is increasing and $\psi ^{\prime }\left( t\right)$ nondecreasing, for all $t\in \mathbb{R}_{+}^{n}$ and for any	$t\geq a$,

\begin{equation*}
u\left( t\right) \leq v\left( t\right) +h\left( t\right) \int_{a}^{t}\psi ^{\prime }\left( s\right) \left( \psi \left( t\right) -\psi \left( s\right) 
\right) ^{\alpha -1}u\left( s\right) ds ;
\end{equation*}

\item $h\left( t\right) $ is nonnegative and nondecreasing.
\end{enumerate}

Then, we have
\begin{equation*}
u\left( t\right) \leq v\left( t\right) \mathbb{E}_{\alpha }\left[ h\left( t\right) \Gamma \left( \alpha \right) \left( \psi \left( t\right) -\psi \left( a\right) 
\right) ^{\alpha }\right],
\end{equation*}
for any $t\geq a$, where $\mathbb{E}_{\alpha}(\cdot)$ is the one-parameter Mittag-Leffler function.
\end{lemma}

\begin{lemma}\textnormal{\cite{ZE1}} \label{LE1} Let $\alpha>0$ and $\delta>0$. If $f(x)= \left( \psi \left( x\right) -\psi \left( a\right)
\right) ^{\delta -1}$, then 
\begin{equation*}
I_{a+}^{\alpha ;\psi }f(x)=\frac{\Gamma \left( \delta \right) }{\Gamma \left(
\alpha +\delta \right) }\left( \psi \left( x\right) -\psi \left( a\right)
\right) ^{\alpha +\delta -1}.
\end{equation*}
\end{lemma}

\section{$\psi$-Hilfer fractional partial derivative}

In this section we introduce the $\psi$-Riemann-Liouville fractional partial integral and the $\psi$-Hilfer fractional partial derivative of a function of 
$N$ variables with respect to another function, as well as some particular cases. Using these definitions, we
introduce new versions of the Ulam-Hyers and Ulam-Hyers-Rassias stabilities. In addition, some important results are discussed.

The first interesting and important result of this paper is the $\psi$-Riemann-Liouville fractional partial integral of a function of $N$ variables with
respect to another function. The motivation for this extension comes from
Eq.(\ref{eq1}).

\begin{definition}
Let $\theta =\left( \theta _{1},\theta _{2},...,\theta  _{N}\right) $ and $\alpha =\left( \alpha _{1},\alpha _{2},...,\alpha  _{N}\right) $, where $0<\alpha _{1},\alpha _{2},...,\alpha _{N}<1$,  $N\in \mathbb{N}$. Also put $\widetilde{I}=I_{1}\times I_{2}\times  \cdot \cdot \cdot \times I_{N}=\left[\theta _{1},a_{1}\right] \times  \left[ \theta _{2},a_{2}\right] \times\cdot \cdot \cdot \times \left[  \theta _{N},a_{N}\right] ,$ where $
a_{1},a_{2},...,a_{N}$ and  $\theta _{1},\theta _{2},...,\theta _{N}$ are positive constants. Also  let $\psi \left( \cdot \right) $ be an increasing and positive monotone  function on $\left( \theta _{1},a_{1}\right] \times
\left(\theta  _{2},a_{2}\right] \times \cdots \times \left( \theta _{N},a_{N}\right]$, having a  continuous derivative $\psi ^{\prime }\left( \cdot \right) $ on $\left(  \theta _{1},a_{1}\right] ,\left( \theta _{2},a_{2}%
\right] ,...,\left(  \theta _{N},a_{N}\right] $.  The $\psi$-Riemann-Liouville partial integral of a function of  $N$ variables $u=\left( u_{1},u_{2},...,u_{N}\right) \in L^{1}\left(  \widetilde{I }\right) $ is defined by
\begin{equation}
I_{\theta ,x}^{\alpha ;\psi }u\left( x\right) =\frac{1}{\Gamma \left( \widetilde{\alpha }_{j}\right) }\int \int \cdot \cdot \cdot \int_{\widetilde{I}}\psi ^{\prime }\left( s_{j}\right) \left( \psi \left( x_{j}\right) -\psi
\left( s_{j}\right) \right) ^{\alpha _{j}-1}u\left( s_{j}\right) ds_{j},\label{eq30}
\end{equation}%
with $\psi ^{\prime }\left( s_{j}\right) \left( \psi \left( x_{j}\right)-\psi \left( s_{j}\right) \right) ^{\alpha _{j}-1}=\psi ^{\prime }\left(s_{1}\right) \left( \psi \left( x_{1}\right) -\psi \left( s_{1}\right)\right) ^{\alpha _{1}-1}\psi ^{\prime }\left( s_{2}\right) \left( \psi\left( x_{2}\right) -\psi \left( s_{2}\right) \right) ^{\alpha _{2}-1}\cdot\cdot \cdot \psi ^{\prime }\left( s_{N}\right) \left( \psi \left(x_{N}\right) -\psi \left( s_{N}\right) \right) ^{\alpha _{N}-1}$ and where we use the notations $\Gamma \left( \alpha _{j}\right) =\Gamma \left( \alpha _{1}\right) \Gamma \left( \alpha _{2}\right) \cdot \cdot \cdot \Gamma \left(\alpha _{N}\right) $, $u\left( s_{j}\right) =u\left( s_{1}\right) u\left(s_{2}\right) \cdot \cdot \cdot u\left( s_{N}\right) $ and $ds_{j}=ds_{1}ds_{2}\cdot \cdot \cdot ds_{N}$, $j\in \left\{1,2,...,N\right\} $ with $N\in \mathbb{N}$.
\end{definition}

Particular cases of this fractional partial integral, Eq.(\ref{eq30}), are presented below. 

\begin{remark}
\mbox{}
\begin{enumerate}
\item If we consider $\psi \left( x_{i}\right) =x_{i}$ in {\rm Eq.(\ref{eq30})}, we obtain the Riemann-Liouville fractional partial integral of a function of $N$ variables, given by {\rm\cite{int2}}
\begin{equation*}
I_{\theta , x}^{\alpha }u\left( x\right) = \frac{1}{\Gamma \left( \alpha_{j}\right) } \int \int \cdot \cdot \cdot \int_{\widetilde{I}}\left( x_{j}-s_{j}\right) ^{\alpha_{j}-1}u \left( s_{j}\right) ds_{j},
\end{equation*}
with $\left( x_{j}-s_{j}\right) ^{\alpha_{j}-1}=\left( x_{1}-s_{1}\right) ^{\alpha _{1}-1}\left(x_{2}-s_{2}\right)^{\alpha _{2}-1}\cdot \cdot \cdot \left(	x_{N}-s_{N}\right) ^{\alpha _{N}-1}$, $\Gamma \left(\alpha _{j}\right) =\Gamma \left( \alpha_{1}\right)	\Gamma \left( \alpha _{2}\right) \cdot \cdot \cdot \Gamma\left( \alpha _{N}\right) ,$ $u\left(s_{j}\right)	=u\left(s_{1}\right) u\left( s_{2}\right) \cdot \cdot \cdot	u\left( s_{N}\right) $and $ds_{j}=ds_{1}ds_{2}\cdot	\cdot \cdot ds_{N}$, $j\in \left\{ 1,2,...,N\right\} $ with $N\in \mathbb{N}$.

\item If we consider $\psi \left( x_{i}\right) =\ln x_{i}$ and $\theta_{i}=0$ for $ i=1,2,...,N$ in {\rm Eq.(\ref{eq30})}, we obtain the Hadamard fractional partial integral of a function of $N$ variables, given by
{\rm\cite{int1}}
\begin{equation*}
I_{\theta ,x}^{\alpha }u\left( x\right) =\frac{1}{\Gamma \left( \alpha _{j}\right) }\int \int \cdot \cdot \cdot \int_{\widetilde{I}_{0}}\left( \ln \frac{x_{j}}{s_{j}}\right) ^{\alpha _{j}-1}u\left( s_{j}\right) \frac{ds_{j}%
}{s_{j}},
\end{equation*}
with $\left( \ln \dfrac{x_{j}}{s_{j}}\right) ^{\alpha _{j}-1}=\left( \ln \dfrac{x_{1}}{s_{1}}\right) ^{\alpha _{1}-1}\left( \ln \dfrac{x_{2}}{s_{2}}\right) ^{\alpha _{2}-1}\cdot \cdot \cdot \left( \ln \dfrac{x_{N}}{s_{N}}%
\right) ^{\alpha _{N}-1}$, $\Gamma \left( \alpha _{j}\right) =\Gamma \left(\alpha _{1}\right) \Gamma \left( \alpha _{2}\right) \cdot \cdot \cdot \Gamma\left( \alpha _{N}\right) $, $u\left( s_{j}\right) =u\left( s_{1}\right)
u\left( s_{2}\right) \cdot \cdot \cdot u\left( s_{N}\right) $ and $ds_{j}=ds_{1}ds_{2}\cdot \cdot \cdot ds_{N}$, $j\in \left\{1,2,...,N\right\} $ with $N\in \mathbb{N}$.

\item If we consider $\psi \left( x_{i}\right) =x_{i}$ and $\theta _{i}=-\infty $ for $i=1,2,...,N$ in \textrm{Eq.(\ref{eq30})}, we obtain the Liouville fractional partial integral of a function of $N$ variables, given by 
\begin{equation*}
I_{\theta ,x}^{\alpha }u\left( x\right) =\frac{1}{\Gamma \left( \alpha _{j}\right) }\int \int \cdot \cdot \cdot \int_{\widetilde{I}_{\infty }}\left( x_{j}-s_{j}\right) ^{\alpha_{j}-1}u\left( s_{j}\right) ds_{j},
\end{equation*}%
with $\left( x_{j}-s_{j}\right) ^{\alpha _{j}-1}=\left( x_{1}-s_{1}\right) ^{\alpha _{1}-1}\left( x_{2}-s_{2}\right) ^{\alpha _{2}-1}\cdot \cdot \cdot \left( x_{N}-s_{N}\right) ^{\alpha _{N}-1}$, $\Gamma \left( \alpha
_{j}\right) =\Gamma \left( \alpha _{1}\right) \Gamma \left( \alpha _{2}\right) \cdot \cdot \cdot \Gamma \left( \alpha _{N}\right) ,$ $u\left(s_{j}\right) =u\left( s_{1}\right) u\left( s_{2}\right) \cdot \cdot \cdot
u\left( s_{N}\right) $ and $ds_{j}=ds_{1}ds_{2}\cdot \cdot \cdot ds_{N}$, $ j\in \left\{ 1,2,...,N\right\} $ with $N\in \mathbb{N}$.

\end{enumerate}
\end{remark}

It is possible to obtain other fractional partial integrals, that is, Erd\`elyi-Kober fractional partial integral, Katugampola fractional partial integral, Weyl fractional partial integral, among others, as well as Sousa and Oliveira \cite{ZE1}, in a recent work, introduced the $\psi$-Hilfer fractional integral, using different choices for $ \psi \left( \cdot \right) $ and parameters $d_{i}$. Moreover, each fractional partial integral obtained here, is an extension of its respective fractional integral {\rm\cite{HER,KSTJ,SAMKO,ZE1}}.

As an application, taking $N=2$ and $\theta_{1}=\theta_{2}=0$ in {\rm {Eq.(\ref{eq30})}} we have the fractional partial integral that will be used in what follows:
\begin{eqnarray}
I_{\theta }^{\alpha ;\psi }u\left( x_{1},x_{2}\right)  &=&\frac{1}{\Gamma \left( \alpha _{1}\right) \Gamma \left( \alpha _{2}\right) }\int_{0}^{x_{1}}\int_{0}^{x_{2}} \psi ^{\prime }\left( s_{1}\right) \psi^{\prime }\left( s_{2}\right) \left( \psi \left( x_{1}\right) -\psi \left(s_{1}\right) \right) ^{\alpha _{1}-1}\notag\label{eq31} \\&&\left( \psi \left( x_{2}\right) -\psi \left( s_{2}\right) \right) ^{\alpha_{2}-1}u\left( s_{1},s_{2}\right) ds_{1}ds_{2},  
\end{eqnarray}
with $0<\alpha _{1},\alpha _{2}\leq 1$.

Also, we have 
\begin{equation}\label{eq32}
I_{0+,x_{1}}^{\alpha _{1};\psi }u\left( x_{1},x_{2}\right) =\frac{1}{\Gamma \left( \alpha _{1}\right) }\int_{0}^{x_{1}}\psi ^{\prime }\left( s_{1}\right) \left( \psi \left( x_{1}\right) -\psi \left( s_{1}\right)
\right) ^{\alpha _{1}-1}u\left( s_{1},s_{2}\right) ds_{1} 
\end{equation}
and
\begin{equation}\label{eq33}
I_{0+,x_{2}}^{\alpha _{2};\psi }u\left( x_{1},x_{2}\right) =\frac{1}{\Gamma \left( \alpha _{2}\right) }\int_{0}^{x_{2}}\psi ^{\prime }\left( s_{2}\right) \left( \psi \left( x_{2}\right) -\psi \left( s_{2}\right)
\right) ^{\alpha _{2}-1}u\left( s_{1},s_{2}\right) ds_{2},
\end{equation}
with $0<\alpha _{1},\alpha _{2}\leq 1$.

Sousa and Oliveira \cite{ZE1} used the Riemann-Liouville fractional integral with respect to a function (one variable) to introduce the $\psi$-Hilfer fractional derivative (one variable); we follow here a similar 
procedure, starting with a fractional integral of $N$ variables, to introduce the $\psi$-Hilfer fractional partial derivative of $N$ variables. 

\begin{definition} Let $\theta =\left(\theta_{1},\theta_{2},...,\theta_{N}\right) $ and $\alpha =\left(
\alpha _{1},\alpha _{2},...,\alpha _{N}\right) $, where	$0<\alpha_{1},\alpha _{2},...,\alpha _{N}<1$, $N\in \mathbb{N}$.  Put $\widetilde{I}=I_{a_{1}}\times I_{a_{2}}\times \cdot \cdot \cdot \times I_{a_{N}}=\left[ \theta_{1},a_{1}\right] \times \left[\theta_{2},a_{2}\right] \times \cdot \cdot \cdot \times \left[
\theta_{N},a_{N}\right]$, where $a_{1},a_{2},...,a_{N}$ and	$\theta_{1}, \theta_{2},..., \theta_{N}$ are positive constants. Also, let functions $u,\psi \in C^{n}\left( \widetilde{I}, \mathbb{R} \right)$ and such that $\psi$ is increasing and $\psi ^{\prime }\left( x_{i}\right) \neq 0,$ $i\in \left\{ 1,2,...,N\right\}$, $x_{i}\in \widetilde{I},$ $N\in  \mathbb{N}$. The $\psi$-Hilfer fractional partial derivative of a function of $N$
variables, of order $\alpha $ and type $0\leq \beta _{1},\beta_{2},...,\beta _{N}\leq 1$, denoted by $^{H}\mathbb{D}_{\theta, \widetilde{x}_{j} }^{\alpha ,\beta ;\psi }\left( \cdot \right)$, is
defined by 
\begin{equation}
^{H}\mathbb{D}_{\theta ,x}^{\alpha ,\beta ;\psi }u\left( x\right) =I_{\theta ,x_{j}}^{\beta \left( 1-\alpha \right) ;\psi }\left( \frac{1}{\psi ^{\prime }\left( x_{j}\right) }\frac{\partial ^{N}}{\partial x_{j}}\right) I_{\theta ,x_{j}}^{\left( 1-\beta \right) \left( 1-\alpha \right) ;\psi }u\left( x_{j}\right) ,  \label{eq34}
\end{equation}
with $\partial x_{j}=\partial x_{1}\partial x_{2}\cdot \cdot \cdot \partial x_{N}$ and $\psi ^{\prime }\left( x_{j}\right) =\psi ^{\prime }\left( x_{1}\right) \psi ^{\prime }\left( x_{2}\right) \cdot \cdot \cdot \psi
^{\prime }\left( x_{N}\right) $, $j\in \left\{ 1,2,...,N\right\} $, $N\in \mathbb{N}$.
\end{definition}

In what follows a few particular cases of fractional partial derivatives are presented. It is important to note that the $\psi$-Hilfer fractional partial derivative contains a wide class of fractional partial derivatives, each one is an extension of the fractional derivative of one variable. 

\begin{remark}

\mbox{}

\begin{enumerate}

\item Taking the limit $\beta \rightarrow 0$ on both sides of \textrm{{Eq.(\ref{eq34})}}, we have the $\psi $-Riemann-Liouville fractional partial derivative of $N$ variables, given by 
\begin{equation*}
^{RL}\mathbb{D}_{\theta ,x}^{\alpha ;\psi }u\left( x\right) =\left( \frac{1}{\psi ^{\prime }\left( x_{j}\right) }\frac{\partial ^{N}}{\partial x_{j}}\right) I_{\theta ,x_{j}}^{1-\alpha ;\psi }u\left( x_{j}\right) ,
\end{equation*}%
with $\partial x_{j}=\partial x_{1}\partial x_{2}\cdot \cdot \cdot \partial x_{N}$ and $\psi ^{\prime }\left( x_{j}\right) =\psi ^{\prime }\left( x_{1}\right) \psi ^{\prime }\left( x_{2}\right) \cdot \cdot \cdot \psi
^{\prime }\left( x_{N}\right) $, $j\in \left\{ 1,2,...,N\right\} $, $N\in  \mathbb{N}$.

\item Taking the limit $\beta \rightarrow 1$ on both sides of \textrm{{Eq.(\ref{eq34})}}, we have the $\psi $-Caputo fractional partial derivative of $N$ variables, given by 
\begin{equation*}
^{C}\mathbb{D}_{\theta ,x}^{\alpha ;\psi }u\left( x\right) =I_{\theta ,x_{j}}^{1-\alpha ;\psi }\left( \frac{1}{\psi ^{\prime }\left( x_{j}\right) } \frac{\partial ^{N}}{\partial x_{j}}\right) u\left( x_{j}\right) ,
\end{equation*}%
with $\partial x_{j}=\partial x_{1}\partial x_{2}\cdot \cdot \cdot \partial x_{N}$ and $\psi ^{\prime }\left( x_{j}\right) =\psi ^{\prime }\left( x_{1}\right) \psi ^{\prime }\left( x_{2}\right) \cdot \cdot \cdot \psi
^{\prime }\left( x_{N}\right) $, $j\in \left\{ 1,2,...,N\right\} $, $N\in \mathbb{N}$.

\item Taking the limit $\beta \rightarrow 0$ on both sides of \textrm{{Eq.(\ref{eq34})}} and choosing $\psi \left( x_{j}\right) =x_{j}$, $j\in \left\{1,2,...,N\right\} $, we have the Riemann-Liouville fractional partial
derivative of $N$ variables, given by \textrm{\cite{int2}} 
\begin{equation*}
^{RL}\mathcal{D}_{\theta ,x}^{\alpha ;\psi }u\left( x\right) =\left( \frac{\partial ^{N}}{\partial x_{j}}\right) I_{\theta ,x_{j}}^{1-\alpha ;\psi}u\left( x_{j}\right) ,
\end{equation*}
with $\partial x_{j}=\partial x_{1}\partial x_{2}\cdot \cdot \cdot \partial x_{N}$, for $N\in \mathbb{N}$.

\item Taking the limit $\beta \rightarrow 1$ on both sides of \textrm{{Eq.(\ref{eq34})}} and choosing $\psi \left( x_{j}\right) =x_{j}$, $j\in \left\{1,2,...,N\right\} $, we have the Caputo fractional partial derivative of $N$ variables, given by \textrm{\cite{B1}} 
\begin{equation*}
^{C}\mathcal{D}_{\theta ,x}^{\alpha ;\psi }u\left( x\right) =I_{\theta ,x_{j}}^{1-\alpha ;\psi }\left( \frac{\partial ^{N}}{\partial x_{j}}\right) u\left( x_{j}\right) ,
\end{equation*}%
with $\partial x_{j}=\partial x_{1}\partial x_{2}\cdot \cdot \cdot \partial x_{N}$, for $N\in \mathbb{N}$.\end{enumerate}
\end{remark}

We have presented only a few particular cases of fractional partial derivatives. It is possible to obtain many other fractional partial derivatives, starting from different choices for function $\psi \left( \cdot \right) $ and taking the limits $\beta \rightarrow 0$ and $\beta \rightarrow 1$, e.g. the Hadamard fractional partial derivative, Caputo-Hadamard fractional partial derivative, Caputo-Katugampola fractional partial derivative, among others. In all these cases, each fractional partial derivative obtained is an extension of its corresponding fractional derivative {\rm\cite{HER,KSTJ,SAMKO,OLMS,ZE1}}.

Now, choosing $N=2$ in {\rm {Eq.(\ref{eq34})}}, we find the fractional partial derivative that will be used in this paper:
\begin{equation}\label{eq35}
^{H}\mathbb{D}_{\theta }^{\alpha ,\beta ;\psi }u\left( x_{1},x_{2}\right) =I_{\theta }^{\beta \left( 1-\alpha \right) ;\psi }\left( \frac{1}{\psi ^{\prime } \left( x_{1}\right) \psi ^{\prime }\left( x_{2}\right) }\frac{\partial ^{2} }{\partial x_{1}\partial x_{2}}\right) I_{\theta }^{\left( 1-\beta \right) 
\left( 1-\alpha \right) ;\psi }u\left( x_{1},x_{2}\right).
\end{equation}
We shall also use the following notation: 
\begin{equation}\label{eq36}
^{H}\mathbb{D}_{\theta }^{\alpha ,\beta ;\psi }u\left( x_{1},x_{2}\right) =\frac{\partial _{\beta ;\psi }^{2\alpha }u}{\partial _{\beta ;\psi }x^{\alpha }\partial_{\beta ;\psi }y^{\alpha }}\left( x_{1},x_{2}\right).
\end{equation}

Let $a,b\in \left( 0,\infty \right] ,$ $\varepsilon >0,$ $\varphi \in C\left( \left[ 0,a\right) \times \left[ 0,b\right), \mathbb{R}_{+}\right)$ and let $(\mathbb{B}, \left\vert \cdot\right\vert)$ be a real or complex Banach space.

Consider also the following inequalities, of paramount importance, which will be used to introduce the Ulam-Hyers and Ulam-Hyers-Rassias stabilities: 
\begin{equation}\label{eq4}
\left\vert \dfrac{\partial _{\beta ;\psi }^{2\alpha }v}{\partial _{\beta ;\psi }^{\alpha }x^{\alpha }\partial _{\beta ;\psi }^{\alpha }y^{\alpha }} 
\left( x,y\right) -f\left( x,y,v\left( x,y\right) ,\frac{\partial _{\beta
;\psi }^{\alpha }v}{\partial _{\beta ;\psi }^{\alpha }x^{\alpha }}\left( x,y\right) ,
\frac{\partial _{\beta ;\psi }^{\alpha }v}{\partial _{\beta ;\psi }^{\alpha }y^{\alpha }}\left( x,y\right) \right) \right\vert \leq
\varepsilon ,
\end{equation}
$x\in \left[ 0,a\right) ,$ $y\in \left[ 0,b\right)$.
\begin{equation}\label{eq5}
\left\vert \frac{\partial _{\beta ;\psi }^{2\alpha }v}{\partial _{\beta ;\psi }^{\alpha }x^{\alpha }\partial _{\beta ;\psi }^{\alpha }y^{\alpha }} 
\left( x,y\right) -f\left( x,y,v\left( x,y\right) ,\frac{\partial _{\beta
;\psi }^{\alpha }v}{\partial _{\beta ;\psi }^{\alpha }x^{\alpha }}\left( x,y\right) ,
\frac{\partial _{\beta ;\psi }^{\alpha }v}{\partial _{\beta ;\psi }^{\alpha }y^{\alpha }}\left( x,y\right) \right) \right\vert \leq
\varphi \left( x,y\right) ,
\end{equation}
$x\in \left[ 0,a\right) ,$ $y\in \left[ 0,b\right)$.
\begin{equation}\label{eq6}
\left\vert \frac{\partial _{\beta ;\psi }^{2\alpha }v}{\partial _{\beta ;\psi }^{\alpha }x^{\alpha }\partial _{\beta ;\psi }^{\alpha }y^{\alpha }} 
\left( x,y\right) -f\left( x,y,v\left( x,y\right) ,\frac{\partial _{\beta
;\psi }^{\alpha }v}{\partial _{\beta ;\psi }^{\alpha }x^{\alpha }}\left( x,y\right) ,\frac{\partial _{\beta ;\psi }^{\alpha }v}{\partial _{\beta ;\psi }^{\alpha }y^{\alpha }}
\left( x,y\right) \right) \right\vert \leq
\varepsilon \varphi \left( x,y\right) ,
\end{equation}
$x\in \left[ 0,a\right) ,$ $y\in \left[ 0,b\right)$, with $\dfrac{1}{2}<\alpha \leq 1$ and $0\leq \beta \leq 1$.

\begin{definition} A function $u$ is a solution of \rm{Eq.(\ref{eq3})}, if
	$u\in C^{1}\left(\left[ 0,a\right) \times \left[ 0,b\right),
	\mathbb{B}\right) $ and $\dfrac{\partial _{\beta ;\psi }^{2\alpha
	}u}{\partial _{\beta ;\psi }^{\alpha }x^{\alpha }\partial _{\beta ;\psi
	}^{\alpha }y^{\alpha }}\left( x,y\right) \in C\left( \left[ 0,a\right)
	\times \left[ 0,b\right), \mathbb{B}\right), $ with $\dfrac{1}{2}
	<\alpha \leq 1$ and $0\leq \beta \leq 1$.
\end{definition}

\begin{definition} A solution of {\rm{Eq.(\ref{eq3})}} admits Ulam-Hyers
	stability if there exist real numbers $C_{f}^{1},C_{f}^{2}$ and $C_{f}^{3}>0$
	such that, for any $\varepsilon >0$ and for any solution $v$ to the
	inequality {\rm{Eq.(\ref{eq4})}}, we have 

\begin{equation*}
\left\vert v\left( x,y\right) -u\left( x,y\right) \right\vert \leq C_{f}^{1}\varepsilon ,\text{ }\forall x\in \left[ 0,a\right) ,\text{ } \forall y\in \left[ 0,b\right) ,
\end{equation*}
\begin{equation*}
\left\vert \dfrac{\partial _{\beta ;\psi }^{\alpha }v}{\partial _{\beta ;\psi }^{\alpha }x^{\alpha }}\left( x,y\right) -
\dfrac{\partial _{\beta ;\psi }^{\alpha }u}{\partial _{\beta ;\psi }^{\alpha }x^{\alpha }}\left(
x,y\right) \right\vert \leq C_{f}^{2}\varepsilon ,\text{ }\forall x\in \left[ 0,a\right) ,\text{ }\forall y\in \left[ 0,b\right) ,
\end{equation*}
\begin{equation*}
\left\vert \dfrac{\partial _{\beta ;\psi }^{\alpha }v}{\partial _{\beta ;\psi }^{\alpha }y^{\alpha }}\left( x,y\right) -
\dfrac{\partial _{\beta ;\psi }^{\alpha }u}{\partial _{\beta ;\psi }^{\alpha }y^{\alpha }}\left(
x,y\right) \right\vert \leq C_{f}^{3}\varepsilon ,\text{ }\forall x\in \left[ 0,a\right) ,\text{ }\forall y\in \left[ 0,b\right) ,
\end{equation*}
with $\dfrac{1}{2}<\alpha \leq 1$ and $0\leq \beta \leq 1$.
\end{definition}

\begin{definition} A solution of {\rm{Eq.(\ref{eq3})}} admits generalized
	Ulam-Hyers-Rassias stability if there exist real numbers $C_{f,\varphi
	}^{1},C_{f,\varphi }^{2}$ and $C_{f,\varphi }^{3}>0$ such that, for any
	$\varepsilon >0$ and for any solution $v$ to the inequality
	{\rm{Eq.(\ref{eq5})}}, we have 
	
\begin{equation*}
\left\vert v\left( x,y\right) -u\left( x,y\right) \right\vert \leq C_{f,\varphi }^{1}\varphi \left( x,y\right) ,\text{ }\forall x\in \left[ 0,a\right) ,
\text{ }\forall y\in \left[ 0,b\right) ,
\end{equation*}
\begin{equation*} \label{eq11}
\left\vert \frac{\partial _{\beta ;\psi }^{\alpha }v}{\partial _{\beta ;\psi }^{\alpha }x^{\alpha }}\left( x,y\right) -
\frac{\partial _{\beta ;\psi }^{\alpha }u}{\partial _{\beta ;\psi }^{\alpha }x^{\alpha }}\left(
x,y\right) \right\vert \leq C_{f,\varphi }^{2}\varphi \left( x,y\right), \text{ }\forall x\in \left[ 0,a\right) ,\text{ }\forall y\in \left[ 0,b\right) ,
\end{equation*}

\begin{equation*}\label{eq12}
\left\vert \frac{\partial _{\beta ;\psi }^{\alpha }v}{\partial _{\beta ;\psi }^{\alpha }y^{\alpha }}\left( x,y\right) -
\frac{\partial _{\beta ;\psi }^{\alpha }u}{\partial _{\beta ;\psi }^{\alpha }y^{\alpha }}\left(
x,y\right) \right\vert \leq C_{f,\varphi }^{3}\varphi \left( x,y\right), \text{ }\forall x\in \left[ 0,a\right) ,\text{ }\forall y\in \left[ 0,b\right) ,
\end{equation*}
with $\dfrac{1}{2}<\alpha \leq 1$ and $0\leq \beta \leq 1$.

\end{definition}

\begin{remark}\label{re3} A function $v$ is a solution to the inequality {\rm{Eq.(\ref{eq4})}} if, and only if, there exists a function 
$g\in C\left( \left[ 0,a\right)\times \left[ 0,b\right), \mathbb{B}\right)$, which depends on $v$, such that 
\begin{enumerate}
\item  For all $\varepsilon >0$, $\left\vert g\left( x,y\right) \right\vert \leq \varepsilon$, for all $x\in \left[ 0,a\right)$, for all $y\in \left[ 0,b\right)$;

\item For all $x\in \left[ 0,a\right) ,$ for all $y\in \left[ 0,b\right)$,
\begin{equation*}
\frac{\partial _{\beta ;\psi }^{2\alpha }v}{\partial _{\beta ;\psi }^{\alpha }x^{\alpha }\partial _{\beta ;\psi }^{\alpha }y^{\alpha }}
\left( x,y\right) =f\left( x,y,v\left( x,y\right) ,\frac{\partial _{\beta ;\psi }^{\alpha }v}{ \partial _{\beta ;\psi }^{\alpha }x^{\alpha }}
\left( x,y\right) ,\frac{ \partial _{\beta ;\psi }^{\alpha }v}{\partial _{\beta ;\psi }^{\alpha }y^{\alpha }}\left( x,y\right) \right) +g\left( x,y\right) 
\end{equation*}
with $\dfrac{1}{2}<\alpha \leq 1$ and $0\leq \beta \leq 1$.
\end{enumerate}
\end{remark}

\begin{remark} A function $v$ is a solution to the inequality  {\rm{Eq.(\ref{eq5})}} if, and only if, there exists a function $g\in 
C\left( \left[ 0,a\right) \times \left[ 0,b\right), \mathbb{B}\right)$, which depends on $v$, such that
\begin{enumerate}
\item  $\left\vert g\left( x,y\right) \right\vert \leq \varphi \left( x,y\right)$, for all $x\in \left[ 0,a\right)$, for all $y\in \left[ 0,b\right)$;

\item For all $x\in \left[ 0,a\right)$, for all $y\in \left[ 0,b\right)$,
\begin{equation*}
\frac{\partial _{\beta ;\psi }^{2\alpha }v}{\partial _{\beta ;\psi }^{\alpha }x^{\alpha }\partial _{\beta ;\psi }^{\alpha }y^{\alpha }}\left( x,y\right) 
=f\left( x,y,v\left( x,y\right) ,\frac{\partial _{\beta ;\psi }^{\alpha }v}{ \partial _{\beta ;\psi }^{\alpha }x^{\alpha }}\left( x,y\right) ,
\frac{ \partial _{\beta ;\psi }^{\alpha }v}{\partial _{\beta ;\psi }^{\alpha }y^{\alpha }}\left( x,y\right) \right) +g\left( x,y\right) 
\end{equation*}
with $\dfrac{1}{2}<\alpha \leq 1$ and $0\leq \beta \leq 1.$
\end{enumerate}
\end{remark}

\begin{remark} A function $v$ is a solution to the inequality  {\rm{Eq.(\ref{eq6})}} if, and only if, there exists a function $g\in 
C\left( \left[ 0,a\right) \times \left[ 0,b\right), \mathbb{B}\right)$, which depends on $v$, such that
\begin{enumerate}

\item  For all $\varepsilon >0$, $\left\vert g\left( x,y\right) \right\vert \leq \varepsilon \varphi \left( x,y\right)$, for all 
$x\in \left[ 0,a\right)$, for all $y\in \left[ 0,b\right)$;

\item For all $x\in \left[ 0,a\right)$, for all $y\in \left[ 0,b\right)$, 
\begin{equation*}
\frac{\partial _{\beta ;\psi }^{2\alpha }v}{\partial _{\beta ;\psi }^{\alpha }x^{\alpha }\partial _{\beta ;\psi }^{\alpha }y^{\alpha }}
\left( x,y\right) =f\left( x,y,v\left( x,y\right) ,\frac{\partial _{\beta ;\psi }^{\alpha }v}{\partial _{\beta ;\psi }^{\alpha }x^{\alpha }}
\left( x,y\right) ,\frac{\partial _{\beta ;\psi }^{\alpha }v}{\partial _{\beta ;\psi }^{\alpha }y^{\alpha }}\left( x,y\right) \right) +g\left( x,y\right) 
\end{equation*}
with $\dfrac{1}{2}<\alpha \leq 1$ and $0\leq \beta \leq 1$.
\end{enumerate}
\end{remark}

We recall that, in this paper, we have used the notation
\begin{equation}\label{eq13}
\begin{array}{ccc}
u_{1}\left( x,y\right) =\dfrac{\partial _{\beta ;\psi }^{\alpha }v}{\partial _{\beta ;\psi }^{\alpha }x^{\alpha }}\left( x,y\right)  &  & u_{2}\left( x,y\right) =\dfrac{\partial _{\beta ;\psi }^{\alpha }v}{\partial _{\beta
;\psi }^{\alpha }x^{\alpha }}\left( x,y\right)  \\ 
&  &  \\ 
v_{1}\left( x,y\right) =\dfrac{\partial _{\beta ;\psi }^{\alpha }v}{\partial _{\beta ;\psi }^{\alpha }x^{\alpha }}\left( x,y\right)  &  & v_{2}\left( x,y\right) =\dfrac{\partial _{\beta ;\psi }^{\alpha }v}{\partial _{\beta
;\psi }^{\alpha }x^{\alpha }}\left( x,y\right) 
\end{array}
\text{ }
\end{equation}
with $\dfrac{1}{2}<\alpha \leq 1$ and $0\leq \beta \leq 1.$

\begin{lemma}\label{LE2} Let $0< \alpha _{1},\alpha_{2},...,\alpha _{N} <1$ and $\delta _{1},\delta _{2},...,\delta _{N}>0$. If $u\left(x_{j}\right) =\left( \psi \left(x_{j}\right)-\psi \left( 0\right) \right) ^{\delta_{j}}-1$ with	$u\left(x\right) =u\left(x_{1},x_{2},...,x_{N}\right)$ and $\left( \psi \left(x_{j}\right) -\psi \left(0\right) \right) ^{\delta_{j}-1}=\left( \psi \left(x_{1}\right) -\psi \left( 0\right) \right)^{\delta _{1}-1}\left( \psi\left( x_{2}\right) -\psi \left( 0\right)\right) ^{\delta _{2}-1}\cdot	\cdot \cdot \left( \psi \left( x_{N}\right) -\psi \left( 0\right)\right) ^{\delta _{N}-1}$, then
\begin{eqnarray}\label{jo}
I_{\theta }^{\alpha ;\psi }u\left(x\right) =\overset{N}{\underset{j=1}{\prod }}\frac{\Gamma \left( \delta_{j}\right) }{\Gamma \left(\alpha_{j}+\delta_{j}\right) }\left(\psi \left(x_{j}\right) -\psi \left( 0\right) \right)^{\alpha_{j}+\delta_{j}-1},
\end{eqnarray}
with $N\in \mathbb{N}$.
\end{lemma} 

\begin{proof} The proof is realized by induction on $N$. According to Lemma \ref{LE1}, Eq.(\ref{jo}) is valid for $N=1$.

For $N=2$, we have
\begin{eqnarray*}
&&I_{\theta }^{\alpha ;\psi }u\left( x_{1},x_{2}\right)  \\
&=& \frac{1}{\Gamma \left( \alpha _{1}\right) \Gamma \left( \alpha _{2}\right) }\int_{0}^{x_{1}}\int_{0}^{x_{2}}\psi ^{\prime }\left( s_{1}\right) \psi^{
\prime }\left( s_{2}\right) \left( \psi \left( x_{1}\right) -\psi \left(s_{1}\right) \right) ^{\alpha _{1}-1}\left( \psi \left( x_{2}\right) -\psi\left( s_{2}\right) 
\right) ^{\alpha _{2}-1}  \nonumber \\
&&\left( \psi \left( s_{1}\right) -\psi \left( 0\right) \right) ^{\delta _{1}-1}\left( \psi \left( s_{2}\right) -\psi \left( 0\right) \right)^{\delta _{2}-1}ds_{2}ds_{1} 
\nonumber \\
&=&\frac{\Gamma \left( \delta _{2}\right) \left( \psi \left( x_{2}\right)-\psi \left( 0\right) \right) ^{\alpha _{2}+\delta _{2}-1}}{\Gamma \left(\alpha _{2}+
\delta _{2}\right) }\int_{0}^{x_{1}}\psi ^{\prime }\left(
s_{1}\right) \left( \psi \left( x_{1}\right) -\psi \left( s_{1}\right)\right) ^{\delta _{1}-1}ds_{1}  \nonumber \\
&=&\frac{\Gamma \left( \delta _{1}\right) \Gamma \left( \delta _{2}\right) }{\Gamma \left( \alpha _{1}+\delta _{1}\right) \Gamma \left( \alpha_{2}+\delta _{2}\right) }
\left( \psi \left( x_{1}\right) -\psi \left(
0\right) \right) ^{\alpha _{1}+\delta _{1}-1}\left( \psi \left( x_{2}\right)-\psi \left( 0\right) \right) ^{\alpha _{2}+\delta _{2}-1}.
\end{eqnarray*}

Suppose it is true for $N-1$, i.e., 
\begin{eqnarray*}
I_{\theta }^{\alpha ;\psi }u\left(x\right)  &=&\frac{\Gamma\left( \delta _{1}\right) }{\Gamma \left( \alpha _{1}+\delta _{1}\right) } \left( \psi \left( x_{1}\right) -\psi \left( 0\right) \right) ^{\alpha
_{1}+\delta _{1}-1}\frac{\Gamma \left( \delta _{2}\right) }{\Gamma \left(\alpha _{2}+\delta _{2}\right) }\left( \psi \left( x_{2}\right) -\psi \left(0\right) \right) ^{\alpha _{2}+\delta _{2}-1}  \nonumber \\
&&\cdot \cdot \cdot \frac{\Gamma \left( \delta _{1}\right) }{\Gamma \left(\alpha _{N-1}+\delta _{N-1}\right) }\left( \psi \left( x_{N-1}\right) -\psi\left( 0\right) \right) ^{\alpha _{N-1}+\delta _{N-1}-1}.
\end{eqnarray*}

Let's prove that it holds for $N$. Indeed, for $u\left(x_{j}\right) =u\left( x_{1},x_{2},...,x_{N}\right) $, $\Gamma \left(\alpha_{j}\right) =\Gamma \left( \alpha _{1}\right) \Gamma\left( \alpha _{2}\right) \cdot \cdot \cdot \Gamma \left( \alpha _{N-1}\right) $, $\psi ^{\prime }\left(s_{j}\right) \left( \psi \left( x_{j}\right) -\psi \left(s_{j}\right)\right) ^{{\alpha }_{j}-1}=\psi ^{\prime }\left( s_{1}\right) \left( \psi \left( x_{1}\right) -\psi \left( s_{1}\right) \right) ^{\alpha _{1}-1}\cdot 
\cdot \cdot \psi ^{\prime }\left( s_{N-1}\right) \left( \psi \left( x_{N-1}\right) -\psi \left( s_{N-1}\right) \right) ^{\alpha _{N-1}-1}$ we have
\begin{eqnarray*}
&&I_{\theta ,x}^{\alpha ;\psi }u\left(x_{j}\right)\\  &=&\frac{1}{\Gamma \left( \alpha_{j}\right) \Gamma\left( 
\alpha _{N}\right) }\int_{0}^{x_{j}}\int_{0}^{x_{N}}\psi^{\prime }\left( s_{j}\right) \left( \psi \left( x_{j}\right) -\psi \left( s_{j}\right) \right) ^{{\alpha }_{j}-1}\left( \psi \left(s_{j}\right) -\psi \left(
0\right) \right) ^{{\delta }_{j}-1} \\
&&\psi ^{\prime }\left( s_{N}\right) \left( \psi \left( x_{N}\right) -\psi\left( s_{N}\right) \right) ^{\alpha _{N}-1}\left( \psi \left( s_{N}\right)-\psi 
\left( 0\right) \right) ^{\delta _{N}-1}ds_{j}ds_{N} \\
&=&\frac{\Gamma \left( \delta _{1}\right) \Gamma \left( \delta _{2}\right)\cdot \cdot \cdot \Gamma \left( \delta _{N-1}\right) }{\Gamma \left( \alpha_{1}+\delta _{1}\right)\Gamma \left( \alpha _{2}+\delta _{2}\right) \cdot
\cdot \cdot \Gamma \left( \alpha _{N-1}+\delta _{N-1}\right) }\frac{1}{\Gamma \left( \alpha _{N}\right) }\left( \psi \left( x_{1}\right) -\psi\left( 0\right) \right) ^{\alpha _{1}+\delta _{1}-1}\times  \\
&&\left( \psi \left( x_{2}\right) -\psi \left( 0\right) \right) ^{\alpha_{2}+\delta _{2}-1}\cdot \cdot \cdot \left( \psi \left( x_{N-1}\right) -\psi\left( 0\right) \right) ^{\alpha _{N-1}+\delta _{N-1}-1}\times  \\
&&\int_{0}^{x_{N}}\psi ^{\prime }\left( s_{N}\right) \left( \psi \left(x_{N}\right) -\psi \left( s_{N}\right) \right) ^{\alpha _{N}-1}\left( \psi\left( s_{N}\right) -\psi \left( 0\right) \right) ^{\delta _{N}-1}ds_{N} \\
&=&\overset{N}{\underset{j=1}{\prod }}\frac{\Gamma \left({\delta }_{j}\right) }{\Gamma \left({\alpha }_{j}+\delta _{j}\right) }\left( \psi \left(x_{j}\right) -\psi \left( 0\right) \right) ^{{\alpha }_{j}+{\delta }_{j}-1}.
\end{eqnarray*}
\end{proof}

\begin{theorem}\label{te2} If $v$ is a solution to the inequality {\textnormal{Eq.(\ref{eq4})}}, then $\left( v,v_{1},v_{2}\right) $ satisfy the following system of integral inequalities:
\begin{eqnarray*}
&&\left\vert 
\begin{array}{c}
v\left( x,y\right) -\dfrac{\left( \psi \left( y\right) -\psi \left( 0\right) \right) ^{\gamma -1}}{\Gamma \left( \gamma \right) }v\left( x,0\right) - 
\dfrac{\left( \psi \left( x\right) -\psi \left( 0\right) \right) ^{\gamma -1} }{\Gamma \left( \gamma \right) }v\left( 0,y\right)  \\ 
+v\left( 0,0\right) -I_{\theta }^{\alpha ;\psi }f\left( x,y,v\left( x,y\right) ,v_{1}\left( x,y\right) ,v_{2}\left( x,y\right) \right) 
\end{array}
\right\vert   \notag \\
&\leq &\varepsilon \dfrac{\left( \psi \left( y\right) -\psi \left( 0\right) \right) ^{\alpha _{2}}}{\Gamma \left( \alpha _{2}+1\right) }\dfrac{
\left( \psi \left( x\right) -\psi \left( 0\right) \right) ^{\alpha _{1}}}{\Gamma \left( \alpha _{1}+1\right) },
\end{eqnarray*}
\begin{eqnarray*}
&&\left\vert v_{1}\left( x,y\right) -\dfrac{\left( \psi \left( y\right) -\psi \left( 0\right) \right) ^{\gamma -1}}{\Gamma \left( \gamma \right) } 
v_{1}\left( x,0\right) -I_{0+,y}^{\alpha _{2};\psi }f\left( x,y,v\left(
x,y\right) ,v_{1}\left( x,y\right) ,v_{2}\left( x,y\right) \right) \right\vert   \notag \\
&\leq &\varepsilon \frac{\left( \psi \left( y\right) -\psi \left( 0\right) \right) ^{\alpha _{2}}}{\Gamma \left( \alpha _{2}+1\right) },
\end{eqnarray*}
\begin{eqnarray*}
&&\left\vert v_{2}\left( x,y\right) -\dfrac{\left( \psi \left( x\right) -\psi \left( 0\right) \right) ^{\gamma -1}}{\Gamma \left( \gamma \right) } 
v_{2}\left( 0,y\right) -I_{0+,x}^{\alpha _{1};\psi }f\left( x,y,v\left(
x,y\right) ,v_{1}\left( x,y\right) ,v_{2}\left( x,y\right) \right) \right\vert   \notag \\
&\leq &\varepsilon \frac{\left( \psi \left( x\right) -\psi \left( 0\right) \right) ^{\alpha _{1}}}{\Gamma \left( \alpha _{1}+1\right) },
\end{eqnarray*}
\end{theorem}

\begin{proof} From Remark \ref{re3} we have that
\begin{eqnarray*}
v\left( x,y\right)  &=&\frac{\left( \psi \left( y\right) -\psi \left( 0\right) \right) ^{\gamma -1}}{\Gamma \left( \gamma \right) }v\left( x,0\right) 
+\frac{\left( \psi \left( x\right) -\psi \left( 0\right) \right)
^{\gamma -1}}{\Gamma \left( \gamma \right) }v\left( 0,y\right)   \notag \\ &&-v\left( 0,0\right) +I_{\theta }^{\alpha ;\psi }f\left( x,y,v\left( x,y\right) 
,v_{1}\left( x,y\right) ,v_{2}\left( x,y\right) \right)
+I_{\theta }^{\alpha ;\psi }g\left( x,y\right) ;
\end{eqnarray*}

\begin{equation*}
v_{1}\left( x,y\right) =\frac{\left( \psi \left( y\right) -\psi \left( 0\right) \right) ^{\gamma -1}}{\Gamma \left( \gamma \right) }v\left( x,0
\right) +I_{0+,y}^{\alpha _{2};\psi }f\left( x,y,v\left( x,y\right)
,v_{1}\left( x,y\right) ,v_{2}\left( x,y\right) \right) +I_{0+,y}^{\alpha _{2};\psi }g\left( x,y\right) ;
\end{equation*}

\begin{equation*}
v_{2}\left( x,y\right) =\frac{\left( \psi \left( x\right) -\psi \left( 0\right) \right) ^{\gamma -1}}{\Gamma \left( \gamma \right) }v\left( 0,y\right) 
+I_{0+,x}^{\alpha _{1};\psi }f\left( x,y,v\left( x,y\right)
,v_{1}\left( x,y\right) ,v_{2}\left( x,y\right) \right) +I_{0+,x}^{\alpha _{1};\psi }g\left( x,y\right) . 
\end{equation*}

The following also hold: 
\begin{eqnarray}\label{eq20}
&&\left\vert 
\begin{array}{c}
v\left( x,y\right) -\dfrac{\left( \psi \left( y\right) -\psi \left( 0\right) \right) ^{\gamma -1}}{\Gamma \left( \gamma \right) }v\left( x,0\right) - 
\dfrac{\left( \psi \left( x\right) -\psi \left( 0\right) \right) ^{\gamma -1} }{\Gamma \left( \gamma \right) }v\left( 0,y\right)  \\ 
+v\left( 0,0\right) -I_{\theta }^{\alpha ;\psi }f\left( x,y,v\left( x,y\right) ,v_{1}\left( x,y\right) ,v_{2}\left( x,y\right) \right) 
\end{array}
\right\vert   \notag \\
&\leq &I_{\theta }^{\alpha ;\psi }\left\vert g\left( x,y\right) \right\vert \leq \varepsilon I_{\theta }^{\alpha ;\psi }\left( 1\right) .
\end{eqnarray}

Note that, by Lemma \ref{LE2}, we have
\begin{eqnarray}\label{A}
I_{\theta }^{\alpha ;\psi }1 &=&\frac{1}{\Gamma \left( \alpha _{1}\right) \Gamma \left( \alpha _{2}\right) }\int_{0}^{x}\psi ^{\prime }\left( s\right)\int_{0}^{y}
\psi ^{\prime }\left( t\right) \left( \psi \left( x\right) -\psi
\left( s\right) \right) ^{\alpha _{1}-1}\left( \psi \left( y\right) -\psi\left( t\right) \right) ^{\alpha _{2}-1}dtds  \notag  \\&=&\frac{\left( \psi \left( y\right) 
-\psi \left( 0\right) \right) ^{\alpha_{2}}\left( \psi \left( x\right) -\psi \left( 0\right) \right) ^{\alpha _{1}}}{\Gamma \left( \alpha _{1}+1\right) 
\Gamma \left( \alpha _{2}+1\right) }.
\end{eqnarray}

Substituting \rm{Eq.(\ref{A})} in \rm{Eq.(\ref{eq20})}, we get
\begin{eqnarray*}
&&\left\vert 
\begin{array}{c}
v\left( x,y\right) -\dfrac{\left( \psi \left( y\right) -\psi \left( 0\right) \right) ^{\gamma -1}}{\Gamma \left( \gamma \right) }v\left( x,0\right) - 
\dfrac{\left( \psi \left( x\right) -\psi \left( 0\right) \right) ^{\gamma -1} }{\Gamma \left( \gamma \right) }v\left( 0,y\right)  \\ 
+v\left( 0,0\right) -I_{\theta }^{\alpha ;\psi }f\left( x,y,v\left( x,y\right) ,v_{1}\left( x,y\right) ,v_{2}\left( x,y\right) \right) 
\end{array}
\right\vert   \notag \\
&\leq &\varepsilon \frac{\left( \psi \left( y\right) -\psi \left( 0\right) \right) ^{\alpha _{2}}\left( \psi \left( x\right) -\psi \left( 0\right) 
\right) ^{\alpha _{1}}}{\Gamma \left( \alpha _{1}+1\right) \Gamma \left(
\alpha _{2}+1\right) }.
\end{eqnarray*}

On the other hand, we also have
\begin{eqnarray*}
&&\left\vert v_{1}\left( x,y\right) -\frac{\left( \psi \left( y\right) -\psi \left( 0\right) \right) ^{\gamma -1}}{\Gamma \left( \gamma \right) }
v_{1}\left( x,0\right) -I_{0+,y}^{\alpha _{2};\psi }f\left( x,y,v\left(
x,y\right) ,v_{1}\left( x,y\right) ,v_{2}\left( x,y\right) \right) \right\vert   \notag \\
&\leq &\varepsilon \frac{\left( \psi \left( y\right) -\psi \left( 0\right) \right) ^{\alpha _{2}}}{\Gamma \left( \alpha _{2}+1\right) }
\end{eqnarray*}
and
\begin{eqnarray*}
&&\left\vert v_{2}\left( x,y\right) -\frac{\left( \psi \left( x\right) -\psi \left( 0\right) \right) ^{\gamma -1}}{\Gamma \left( \gamma \right) } 
v_{2}\left( 0,y\right) -I_{0+,x}^{\alpha _{1};\psi }f\left( x,y,v\left(
x,y\right) ,v_{1}\left( x,y\right) ,v_{2}\left( x,y\right) \right) \right\vert   \notag \\
&\leq &\varepsilon \frac{\left( \psi \left( x\right) -\psi \left( 0\right) \right) ^{\alpha _{1}}}{\Gamma \left( \alpha _{1}+1\right) }.
\end{eqnarray*}

\end{proof}

The proofs of the following theorems will be optimized here, since they follow the same reasoning of the previous one.

\begin{theorem}\label{te3} If $v$ is a solution to the inequality \textnormal{Eq.(\ref{eq5})}, then $\left( v,v_{1},v_{2}\right) $ satisfy the following integral inequalities:
\begin{eqnarray*}
&&\left\vert 
\begin{array}{c}
v\left( x,y\right) -\dfrac{\left( \psi \left( y\right) -\psi \left( 0\right) \right) ^{\gamma -1}}{\Gamma \left( \gamma \right) }v\left( x,0\right) - \dfrac{\left( \psi \left( x\right) -\psi \left( 0\right) \right) ^{\gamma -1} }{\Gamma \left( \gamma \right) }v\left( 0,y\right)  \\ 
+v\left( 0,0\right) -I_{\theta }^{\alpha ;\psi }f\left( x,y,v\left( x,y\right) ,v_{1}\left( x,y\right) ,v_{2}\left( x,y\right) \right) 
\end{array}
\right\vert   \notag \\ &\leq &I_{\theta }^{\alpha ;\psi }\varphi \left( x,y\right) ,
\end{eqnarray*}

\begin{eqnarray*}
&&\left\vert v_{1}\left( x,y\right) -\frac{\left( \psi \left( y\right) -\psi \left( 0\right) \right) ^{\gamma -1}}{\Gamma \left( \gamma \right) } 
v_{1}\left( x,0\right) -I_{0+,y}^{\alpha _{2};\psi }f\left( x,y,v\left(
x,y\right) ,v_{1}\left( x,y\right) ,v_{2}\left( x,y\right) \right) \right\vert   \notag \\
&\leq &I_{0+,y}^{\alpha _{2};\psi }\varphi \left( x,y\right) ,
\end{eqnarray*}

\begin{eqnarray*}
&&\left\vert v_{2}\left( x,y\right) -\frac{\left( \psi \left( x\right) -\psi \left( 0\right) \right) ^{\gamma -1}}{\Gamma \left( \gamma \right) } 
v_{2}\left( 0,y\right) -I_{0+,x}^{\alpha _{1};\psi }f\left( x,y,v\left(
x,y\right) ,v_{1}\left( x,y\right) ,v_{2}\left( x,y\right) \right) \right\vert   \notag \\
&\leq &I_{0+,x}^{\alpha _{1};\psi }\varphi \left( x,y\right) ,
\end{eqnarray*}
with $\dfrac{1}{2}<\alpha \leq 1,$ $0\leq \gamma <1$ $\left( \gamma =\alpha +\beta \left( \alpha -1\right) \right)$, for all $x\in \left[ 0,a\right) $ 
and $y\in \left[ 0,b\right)$.
\end{theorem}

\begin{theorem}\label{te4} If $v$ is a solution to the inequality \textnormal{Eq.(\ref{eq6})}, then $\left( v,v_{1},v_{2}\right) $ satisfy the following integral inequalities:
\begin{eqnarray*}
&&\left\vert 
\begin{array}{c}
v\left( x,y\right) -\dfrac{\left( \psi \left( y\right) -\psi \left( 0\right) \right) ^{\gamma -1}}{\Gamma \left( \gamma \right) }v\left( x,0\right) - \dfrac{\left( \psi \left( x\right) -\psi \left( 0\right) \right) ^{\gamma -1}}{\Gamma \left( \gamma \right) }v\left( 0,y\right)  \\ 
+v\left( 0,0\right) -I_{\theta }^{\alpha ;\psi }f\left( x,y,v\left( x,y\right) ,v_{1}\left( x,y\right) ,v_{2}\left( x,y\right) \right) 
\end{array}
\right\vert   \notag \\ &\leq &\varepsilon I_{\theta }^{\alpha ;\psi }\varphi \left( x,y\right),
\end{eqnarray*}

\begin{eqnarray*}
&&\left\vert v_{1}\left( x,y\right) -\frac{\left( \psi \left( y\right) -\psi \left( 0\right) \right) ^{\gamma -1}}{\Gamma \left( \gamma \right) } v_{1}\left( x,0\right) -I_{0+,y}^{\alpha _{2};\psi }f\left( x,y,v\left(
x,y\right) ,v_{1}\left( x,y\right) ,v_{2}\left( x,y\right) \right) \right\vert   \notag \\
&\leq &\varepsilon I_{0+,y}^{\alpha _{2};\psi }\varphi \left( x,y\right),
\end{eqnarray*}

\begin{eqnarray*}
&&\left\vert v_{2}\left( x,y\right) -\frac{\left( \psi \left( x\right) -\psi \left( 0\right) \right) ^{\gamma -1}}{\Gamma \left( \gamma \right) } v_{2}\left( 0,y\right) -I_{0+,x}^{\alpha _{1};\psi }f\left( x,y,v\left(
x,y\right) ,v_{1}\left( x,y\right), v_{2}\left( x,y\right) \right) \right\vert   \notag \\
&\leq &\varepsilon I_{0+,x}^{\alpha _{1};\psi }\varphi \left( x,y\right), 
\end{eqnarray*}
with $\dfrac{1}{2}<\alpha \leq 1,$ $0\leq \gamma <1$ $\left( \gamma =\alpha +\beta \left( \alpha -1\right) \right) ,$ for all $x\in \left[ 0,a\right) $ and $y\in \left[ 0,b\right)$.
\end{theorem}

\section{Ulam-Hyers stability} 

In this section we present a result on the existence and uniqueness of the solution to \rm {Eq.(\ref{eq3})} and we deduce a result on Ulam-Hyers stability for the same equation in the case \ $a<\infty $ and $b<\infty$.

\begin{theorem}\label{ver1}We assume that

\textnormal{(i)} $a<\infty $ and $b<\infty$;

\textnormal{(ii)} $f\in C\left( \left[ 0,a\right] \times \left[ 0,b\right] \times \mathbb{B}^{3}, \mathbb{B}\right)$;

\textnormal{(iii)} There exists $L_{f}>0$ such that
\begin{equation*}
\left\vert f\left( x,y,z_{1},z_{2},z_{3}\right) -f\left( x,y,t_{1},t_{2},t_{3}\right) \right\vert \leq L_{f}\underset{i\in \left\{ 1,2,3\right\} }{\max }
\left\vert z_{i}-t_{i}\right\vert 
\end{equation*}
for all $x\in \left[ 0,a\right] ,$ $y\in \left[ 0,b\right] $ and $z_{1},z_{2},z_{3},t_{1},t_{2},t_{3}\in \mathbb{B}$.

Then
\begin{enumerate}

\item  For $\phi \in C^{1}\left( \left[ 0,a\right], \mathbb{B}\right) $ and $\xi \in C^{1}\left( \left[ 0,b\right], \mathbb{B}\right)$, \textnormal{Eq.(\ref{eq3})} has a unique solution satisfying \begin{equation}\label{eq37}
\left\{ 
\begin{array}{ccc}
I_{\theta }^{1-\gamma ;\psi }u\left( x,0\right)  &=&\phi \left( x\right),\text{ }\forall x\in \left[ 0,a\right] \\
I_{\theta }^{1-\gamma ;\psi }u\left( 0,y\right)  &=&\xi \left( y\right),\text{ }\forall y\in \left[ 0,b\right]
\end{array}
\right.
\end{equation}
with $\gamma =\alpha +\beta \left( \alpha -1\right) $ $\left( 0\leq \gamma <1\right)$.

\item The solution of \textnormal{Eq.(\ref{eq3})} is Ulam-Hyers stable.
\end{enumerate}
\end{theorem}

\begin{proof}

\textnormal{1.} Let us consider problem \textnormal{Eq.(\ref{eq3})} with conditions \rm {Eq.(\ref{eq37})} as a fixed point problem. If $u$ is a solution of the problem \rm {Eq.(\ref{eq3})} and \rm {Eq.(\ref{eq37})}, then $\left( u, \dfrac{\partial _{\beta ;\psi }^{\alpha }u}{\partial _{\beta ;\psi}x^{\alpha }}, \dfrac{\partial _{\beta ;\psi }^{\alpha }u}{\partial_{\beta ;\psi }y^{\alpha }}\right) $ is a solution of the following system: 
\begin{equation}\label{eq38}
\left\{ 
\begin{array}{ccc}
u\left( x,y\right)  & = & 
\begin{array}{c}
\dfrac{\left( \psi \left( y\right) -\psi \left( 0\right) \right) ^{\gamma -1}}{\Gamma \left( \gamma \right) }\varphi \left( x\right) +\dfrac{\left( \psi\left( x\right) -\psi \left( 0\right) \right) ^{\gamma -1}}{\Gamma \left(\gamma \right) }\xi \left( y\right) -\varphi \left( 0\right)  \\ 
+I_{\theta }^{\alpha ;\psi }f\left( x,y,u\left( x,y\right) ,u_{1}\left(x,y\right) ,u_{2}\left( x,y\right) \right) 
\end{array}
\\ 
u_{1}\left( x,y\right)  & = & \dfrac{\left( \psi \left( y\right) -\psi \left(0\right) \right) ^{\gamma -1}}{\Gamma \left( \gamma \right) }\varphi \left(x\right) +I_{0+,y}^{\alpha _{2};\psi }f\left( x,y,u\left( x,y\right)
,u_{1}\left( x,y\right) ,u_{2}\left( x,y\right) \right)  \\ 
u_{2}\left( x,y\right)  & = & \dfrac{\left( \psi \left( x\right) -\psi \left(0\right) \right) ^{\gamma -1}}{\Gamma \left( \gamma \right) }\xi \left(y\right) +I_{0+,x}^{\alpha _{1};\psi }f\left( x,y,u\left( x,y\right)
,u_{1}\left( x,y\right) ,u_{2}\left( x,y\right) \right) 
\end{array}
\right. 
\end{equation}
with $\dfrac{1}{2}<\alpha _{1},\alpha _{2}\leq 1$ and $0\leq \gamma <1$. We can write this in a general form:

\begin{equation*}
\left\{ 
\begin{array}{ccc}
u\left( x,y\right) &=&A_{1}\left( u,u_{1},u_{2}\right) \left( x,y\right)\\
u_{1}\left( x,y\right) &=&A_{2}\left( u,u_{1},u_{2}\right) \left( x,y\right)\\ 
u_{2}\left( x,y\right) &=&A_{3}\left( u,u_{1},u_{2}\right) \left( x,y\right)
\end{array}
\right.
\end{equation*}
$u,u_{1},u_{2}\in C\left( \left[ 0,a\right] \times \left( 0,b\right) \right)$.

If $\left( u,u_{1},u_{2}\right) \in C\left( \left[ 0,a\right] \times \left[ 0,b\right] \right) ^{3}$ is a solution of \rm {Eq.(\ref{eq38})}, then $u\in C^{1}\left( \left[ 0,a\right] \times \left[ 0,b\right] \right) $ and $u_{1}\left( x,y\right) = \dfrac{ \partial _{\beta ;\psi }^{\alpha }u}{\partial _{\beta ;\psi }x^{\alpha }}\left( x,y\right) $, $u_{2}\left( x,y\right) =\dfrac{\partial _{ \beta ;\psi }^{\alpha }u}{\partial _{\beta ;\psi }y^{\alpha }}\left( x,y\right) $ i.e., $v$ is solution of \rm {Eq.(\ref{eq3})} and \rm {Eq.(\ref{eq37})}.

Let $X:=C\left( \left[ 0,a\right] \times \left[ 0,b\right] \right) \times C\left( \left[ 0,a\right] \times \left[ 0,b\right] \right) \times C
\left( \left[ 0,a\right] \times \left[ 0,b\right] \right) $ and 
\begin{equation*}
\left\Vert \left( u,u_{1},u_{2}\right) \right\Vert :=\max \left\{ 
\begin{array}{c}
\underset{\left( x,y\right) \in \left[ 0,a\right] \times \left[ 0,b\right] }{\max }\left\vert u\left( x,y\right) \right\vert \exp (-\tau \left(x+y\right) ), \\ 
\underset{\left( x,y\right) \in \left[ 0,a\right] \times \left[ 0,b\right] }{\max }\left\vert u_{1}\left( x,y\right) \right\vert \exp (-\tau \left(x+y\right) ), \\ 
\underset{\left( x,y\right) \in \left[ 0,a\right] \times \left[ 0,b\right] }{\max }\left\vert u_{2}\left( x,y\right) \right\vert \exp (-\tau \left(x+y\right) )
\end{array}
\right\};
\end{equation*}
$\left( C\left( \left[ 0,a\right] \times \left[ 0,b\right] \textbf{}\right);\left\Vert \cdot \right\Vert _{B}\right) $ is a Banach space.

Let $A:X\rightarrow X,$ $\left( u,u_{1},u_{2}\right) \rightarrow \left( A_{1}\left( u,u_{1},u_{2}\right) ,A_{2}\left( u,u_{1},u_{2}\right),A_{3}
\left( u,u_{1},u_{2}\right) \right) $; we then have
\begin{equation*}
\left\Vert A_{1}\left( \overline{u},\overline{u_{1}},\overline{u_{2}}\right)-A_{1}\left(\overline{\overline{u}},\overline{\overline{u_{1}}},\overline{ \overline{u_{2}}}\right) \right\Vert _{B}\leq \frac{L_{f}}{\tau }\left\Vert \left( \overline{u},\overline{u_{1}},\overline{u_{2}}\right) -\left( 
\overline{\overline{u}},\overline{\overline{u_{1}}},\overline{\overline{u_{2}}}\right) \right\Vert _{B}.
\end{equation*}

Thus, if $\tau >0$ is such that $\dfrac{L_{f}}{\tau }<1$, then operator $A$ is a contraction and $A$ is a Picard operator. Therefore, \rm {Eq.(\ref{eq3})} and \rm {Eq.(\ref{eq37})} have a unique solution.

2. Let $v$ be a solution to the inequality \rm {Eq.(\ref{eq4})} and let $u$ be the unique solution to \rm {Eq.(\ref{eq3})}, which satisfies the following conditions: 
\begin{equation}\label{eq45}
\left\{ 
\begin{array}{ccc}
I_{\theta }^{1-\gamma ;\psi }u\left( x,0\right)  &=&v\left( x,0\right),\text{ }\forall x\in \left[ 0,a\right]  \\  I_{\theta }^{1-\gamma ;\psi }u\left( 0,y\right)  
&=&u\left( 0,y\right),\text{ }\forall y\in \left[ 0,b\right] \end{array}
\right. .
\end{equation}

From Theorem \ref{te2}, hypothesis (iii) and from Gronwall Lemma \ref{l1}, it follows that
\begin{eqnarray*}
&&\left\vert v\left( x,y\right) -u\left( x,y\right) \right\vert   \notag
\label{eq46} \\
&=&\left\vert 
\begin{array}{c}
v\left( x,y\right) -\dfrac{\left( \psi \left( y\right) -\psi \left( 0\right)\right) ^{\gamma -1}}{\Gamma \left( \gamma \right) }v\left( x,0\right) -
\dfrac{\left( \psi \left( x\right) -\psi \left( 0\right) \right) ^{\gamma -1}%
}{\Gamma \left( \gamma \right) }v\left( 0,y\right)  \\ +v\left( 0,0\right) -I_{\theta }^{\alpha ;\psi }f\left( x,y,u\left(x,y\right) ,u_{1}\left( x,y\right) ,
u_{2}\left( x,y\right) \right) 
\end{array}
\right\vert   \notag \\
&\leq &\left\vert 
\begin{array}{c}
v\left( x,y\right) -\dfrac{\left( \psi \left( y\right) -\psi \left( 0\right)\right) ^{\gamma -1}}{\Gamma \left( \gamma \right) }v\left( x,0\right) -
\dfrac{\left( \psi \left( x\right) -\psi \left( 0\right) \right) ^{\gamma -1}%
}{\Gamma \left( \gamma \right) }v\left( 0,y\right)  \\+v\left( 0,0\right) -I_{\theta }^{\alpha ;\psi }f\left( x,y,v\left(x,y\right) ,v_{1}\left( x,y\right) ,
v_{2}\left( x,y\right) \right) 
\end{array}
\right\vert   \notag \\
&&+I_{\theta }^{\alpha ;\psi }\left\vert f\left( x,y,v\left( x,y\right),v_{1}\left( x,y\right) ,v_{2}\left( x,y\right) \right) -f\left( x,y,u\left(x,y\right) ,
u_{1}\left( x,y\right) ,u_{2}\left( x,y\right) \right)
\right\vert   \notag \\
&\leq &\varepsilon \dfrac{\left( \psi \left( x\right) -\psi \left( 0\right)\right) ^{\alpha _{1}}\left( \psi \left( y\right) -\psi \left( 0\right)
\right) ^{\alpha _{2}}}{\Gamma \left( \alpha _{1}+1\right) \Gamma \left(
\alpha _{2}+1\right) }+L_{f}I_{\theta }^{\alpha ;\psi }\left\{ \underset{i\in \left\{ 1,2,3\right\} }{\max }\left\vert v_{i}\left( x,y\right)-u_{i}
\left( x,y\right) \right\vert \right\}   \notag \\
&\leq &\varepsilon \frac{\left( \psi \left( a\right) -\psi \left( 0\right)\right) ^{\alpha _{1}}\left( \psi \left( b\right) -\psi \left( 0\right)
\right) ^{\alpha _{2}}}{\Gamma \left( \alpha _{1}+1\right) \Gamma \left(
\alpha _{2}+1\right) }\times   \notag \\&&\mathbb{E}_{\alpha }\left[ L_{f}\Gamma \left( \alpha _{1}\right) \Gamma
\left( \alpha _{2}\right) \left( \psi \left( a\right) -\psi \left( 0\right)\right) ^{\alpha _{1}}\left( \psi \left( b\right) -\psi \left( 0\right)
\right) ^{\alpha _{2}}\right]   \notag \\&=&\varepsilon C_{f}^{1}
\end{eqnarray*}
where
\begin{equation*}
C_{f}^{1}:=\frac{\left( \psi \left( a\right) -\psi \left( 0\right) \right) ^{\alpha _{1}}\left( \psi \left( b\right) -\psi \left( 0\right) 
\right) ^{\alpha _{2}}}{\Gamma \left( \alpha _{1}+1\right) \Gamma \left( \alpha
_{2}+1\right) }\mathbb{E}_{\alpha }\left[ L_{f}\Gamma \left( \alpha _{1}\right) \Gamma \left( \alpha _{2}\right) \left( \psi \left( a\right) -\psi 
\left( 0\right) \right) ^{\alpha _{1}}\left( \psi \left( b\right) -\psi \left(
0\right) \right) ^{\alpha _{2}}\right],
\end{equation*}
with ${\mathbb{E}}_{\alpha}(\cdot)$ is the one-parameter Mittag-Leffler function.

Similarly, we get,
\begin{eqnarray*}
&&\left\vert v_{1}\left( x,y\right) -u_{1}\left( x,y\right) \right\vert  \notag  \\
&\leq &\left\vert v_{1}\left( x,y\right) -\frac{\left( \psi \left( y\right)-\psi \left( 0\right) \right) ^{\gamma -1}}{\Gamma \left( \gamma \right) }v_{1}
\left( x,0\right) -I_{0+,y}^{\alpha _{2};\psi }f\left( x,y,v\left(
x,y\right) ,v_{1}\left( x,y\right) ,v_{2}\left( x,y\right) \right)\right\vert   \notag \\
&&+I_{0+,y}^{\alpha _{2};\psi }\left\vert f\left( x,y,v\left( x,y\right),v_{1}\left( x,y\right) ,v_{2}\left( x,y\right) \right) -f\left( x,y,u\left(x,y\right) ,
u_{1}\left( x,y\right) ,u_{2}\left( x,y\right) \right)
\right\vert   \notag \\
&\leq &\varepsilon \frac{\left( \psi \left( y\right) -\psi \left( 0\right)\right) ^{\alpha _{2}}}{\Gamma \left( \alpha _{2}+1\right) }+L_{f}I_{0+,y}^{\alpha _{2};
\psi }\left\{ \underset{i\in \left\{1,2,3\right\} }{\max }\left\vert v_{i}\left( x,y\right) -u_{i}\left(x,y\right) \right\vert \right\}   \notag \\
&\leq &\varepsilon \frac{\left( \psi \left( b\right) -\psi \left( 0\right)\right) ^{\alpha _{2}}}{\Gamma \left( \alpha _{2}+1\right) }\mathbb{E}_{\alpha }
\left[ L_{f}\Gamma \left( \alpha _{2}\right) \left( \psi \left(
b\right) -\psi \left( 0\right) \right) ^{\alpha _{2}}\right]   \notag \\&=&\varepsilon C_{f}^{2}
\end{eqnarray*}
where
\begin{equation*}
C_{f}^{2}:=\frac{\left( \psi \left( b\right) -\psi \left( 0\right) \right) ^{\alpha _{2}}}{\Gamma \left( \alpha _{2}+1\right) }\mathbb{E}_{\alpha }\left[ L_{f}
\Gamma \left( \alpha _{2}\right) \left( \psi \left( b\right) -\psi \left( 0\right) \right) ^{\alpha _{2}}\right]
\end{equation*}
and
\begin{eqnarray*}
&&\left\vert v_{2}\left( x,y\right) -u_{2}\left( x,y\right) \right\vert  \\
&\leq &\left\vert v_{2}\left( x,y\right) -\frac{\left( \psi \left( x\right)-\psi \left( 0\right) \right) ^{\gamma -1}}{\Gamma \left( \gamma \right) }v_{2}\left( 0,y
\right) -I_{0+,x}^{\alpha _{1};\psi }f\left( x,y,v\left(
x,y\right) ,v_{1}\left( x,y\right) ,v_{2}\left( x,y\right) \right)\right\vert  \\
&&+I_{0+,x}^{\alpha _{1};\psi }\left\vert f\left( x,y,v\left( x,y\right),v_{1}\left( x,y\right) ,v_{2}\left( x,y\right) \right) -f\left( x,y,u\left(x,y\right) ,
u_{1}\left( x,y\right) ,u_{2}\left( x,y\right) \right)\right\vert  \\
&\leq &\varepsilon \frac{\left( \psi \left( x\right) -\psi \left( 0\right)\right) ^{\alpha 1}}{\Gamma \left( \alpha _{1}+1\right) }+L_{f}I_{0+,x}^{\alpha _{1};\psi }
\left\{ \underset{i\in \left\{ 1,2,3\right\} }{\max }\left\vert v_{i}\left( x,y\right) -u_{i}\left( x,y\right) \right\vert\right\}  \\
&\leq &\varepsilon \frac{\left( \psi \left( b\right) -\psi \left( 0\right)\right) ^{\alpha _{1}}}{\Gamma \left( \alpha _{1}+1\right) }\mathbb{E}_{\alpha }\left[ L_{f}
\Gamma \left( \alpha _{1}\right) \left( \psi \left(
a\right) -\psi \left( 0\right) \right) ^{\alpha _{1}}\right]  \\&=&\varepsilon C_{f}^{3}
\end{eqnarray*}
where
\begin{equation*}
C_{f}^{3}:=\frac{\left( \psi \left( a\right) -\psi \left( 0\right) \right) ^{\alpha _{1}}}{\Gamma \left( \alpha _{1}+1\right) }\mathbb{E}_{\alpha }\left[ L_{f}\Gamma 
\left( \alpha _{1}\right) \left( \psi \left( a\right) -\psi \left( 0\right) \right) ^{\alpha _{1}}\right].
\end{equation*}

\end{proof}

\begin{remark}
\mbox{ }
\begin{enumerate}
\item In general, if $a=\infty $ or $b=\infty$, then the solution of {\rm{Eq.(\ref{eq3})}} is not Ulam-Hyers stable;

\item Taking the limit $\beta \rightarrow 0$ on both sides of {\rm {Eq.(\ref{eq3})}}, we have a hyperbolic fractional partial differential equation in the $\psi$-Riemann-Liouville sense.  Consequently, {\rm Theorem \ref{ver1}} is valid for the $\psi$-Riemann-Liouville fractional derivative.

\item Taking the limit $\beta \rightarrow 1$ on both sides of {\rm {Eq.(\ref{eq3})}}, we have a hyperbolic fractional partial differential equation in the $\psi$-Caputo sense. Consequently, {\rm Theorem \ref{ver1}} is valid for the  $\psi$-Caputo fractional derivative;

\item Note that if we take the limits $\beta \rightarrow 1$ or $\beta\rightarrow 0$, we find in both cases a hyperbolic fractional partial differential equation in the sense of the respective fractional derivative. We also have freedom to choose a function $\psi\left( \cdot \right)$. Then, from the choice of function $\psi\left( \cdot \right)$ together with one of the limits, we can obtain a class of hyperbolic fractional partial differential equations for which {\rm Theorem \ref{ver1}} holds;

\item Taking the limit $\alpha =\left( \alpha _{1},\alpha _{2}\right)\rightarrow 1$ and choosing $\psi \left( t\right) =t$ in {\rm {Eq.(\ref{eq3})}}, we obtain a hyperbolic fractional partial differential equation; consequently, the above theorem is valid, with $C_{f}^{1}=ba\exp \left(L_{f}ba\right) $, $C_{f}^{2}=b\exp \left( L_{f}b\right) $ and	$C_{f}^{3}=a\exp \left( L_{f}a\right)$.

\end{enumerate}
\end{remark}

\section{Generalized Ulam-Hyers-Rassias stability}

In this section, we prove the generalized Ulam-Hyers-Rassias stability of the hyperbolic fractional partial differential equation \rm {Eq.(\ref{eq3})}. We consider \rm {Eq.(\ref{eq3})} and inequality \rm {Eq.(\ref{eq5})} in the case $a=\infty $ and $b=\infty$.

\begin{theorem}\label{ver2} We assume that: 
\begin{enumerate}

\item $f \in C\left( \left[ 0,\infty \right) \times \left[ 0,\infty \right) \times 	\mathbb{B}^{3}, \mathbb{B}\right)$;

\item  There exists $L_{f}\in C^{1}\left( \left[ 0,\infty \right) \times \left[0,\infty \right),\mathbb{R}_{+}\right) $ such that
\begin{equation*}
\left\vert f\left( x,y,z_{1},z_{2},z_{3}\right) -f\left( x,y,t_{1},t_{2},t_{3}\right) \right\vert \leq L_{f}\left( x,y\right) \underset{i\in \left\{ 1,2,3\right\} }{\max }\left\vert z_{i}-t_{i}\right\vert
\end{equation*}
for all $x,y\in \left[ 0,\infty \right)$.

\item There exist $\lambda _{\varphi }^{1},\lambda _{\varphi }^{2},\lambda _{\varphi }^{3}>0$ such that
\begin{equation*}
\left\{ 
\begin{array}{ccc}
I_{\theta }^{\alpha ;\psi }\varphi \left( x,y\right) \leq \lambda _{\varphi }^{1}\varphi \left( x,y\right) \text{, }\forall x,y\in \left[ 0,\infty \right)  \\ 
I_{0+,y}^{\alpha _{2};\psi }\varphi \left( x,y\right) \leq \lambda _{\varphi }^{2}\varphi \left( x,y\right) ,\text{ }\forall x,y\in \left[ 0,\infty \right)\\
I_{0+,x}^{\alpha _{1};\psi }\varphi \left( x,y\right) \leq \lambda _{\varphi }^{3}\varphi \left( x,y\right) \text{, }\forall x,y\in \left[ 0,\infty \right)
.\end{array}
\right.
\end{equation*}

\item $\varphi :\mathbb{R}_{+}\times \mathbb{R}_{+}\rightarrow \mathbb{R} _{+}$ is increasing.

\end{enumerate}

Then the solution of {\rm {Eq.(\ref{eq3})}} $\left( a=\infty \text{ and }b=\infty \right) $ is generalized Ulam-Hyers-Rassias stable.
\end{theorem}

\begin{proof} Let $v$ be a solution to the inequality {\rm {Eq.(\ref{eq5})}}. Denote by $u$ the unique solution of the fractional Darboux problem: 
\begin{equation}\label{eq56}
\left\{ 
\begin{array}{ccl}
\dfrac{\partial _{\beta ;\psi }^{2\alpha }u}{\partial _{\beta ;\psi }x^{\alpha }\partial _{\beta ;\psi }y^{\alpha }}\left( x,y\right)  &=&
f\left( x,y,u\left( x,y\right) ,u_{1}\left( x,y\right) ,u_{2}\left(
x,y\right) \right) ,\text{ }\forall x,y\in \left[ 0,\infty \right) \\
I_{\theta }^{1-\gamma ;\psi }u\left( x,0\right)  &=&v\left( x,0\right), \text{ }\forall x,y\in \left[ 0,\infty \right) \\ 
I_{\theta }^{1-\gamma ;\psi }u\left( 0,y\right)  &=&v\left( 0,y\right), \text{ }\forall x,y\in \left[ 0,\infty \right) 
.\end{array}
\right.
\end{equation}

If $u$ is a solution for the fractional problem {\rm Eq.(\ref{eq56})}, then $\left( u,u_{1},u_{2}\right) $ is a solution of the following problem: 
\begin{eqnarray}
u\left( x,y\right)  &=&\frac{\left( \psi \left( y\right) -\psi \left(0\right) \right) ^{\gamma -1}}{\Gamma \left( \gamma \right) }v\left(x,0\right) + \frac{\left( \psi \left( x\right) -\psi \left( 0\right) \right)^{\gamma -1}}{\Gamma \left( \gamma \right) }v\left( 0,y\right) -v\left(0,0\right) \notag  \label{eq57} \\
&&+I_{\theta }^{\alpha ;\psi }f\left( x,y,u\left( x,y\right) ,u_{1}\left(x,y\right) ,u_{2}\left( x,y\right) \right) 
\end{eqnarray}
\begin{equation}\label{eq58}
u_{1}\left( x,y\right) =\frac{\left( \psi \left( y\right) -\psi \left( 0\right) \right) ^{\gamma -1}}{\Gamma \left( \gamma \right) }v_{1}\left( x,0\right) + I_{0+,y}^{\alpha _{2};\psi }f\left( x,y,u\left( x,y\right)
,u_{1}\left( x,y\right) ,u_{2}\left( x,y\right) \right) 
\end{equation}
\begin{equation}\label{eq59}
u_{2}\left( x,y\right) =\frac{\left( \psi \left( x\right) -\psi \left( 0\right) \right) ^{\gamma -1}}{\Gamma \left( \gamma \right) }v_{2}\left( 0,y\right) + I_{0+,x}^{\alpha _{1};\psi }f\left( x,y,u\left( x,y\right)
,u_{1}\left( x,y\right) ,u_{2}\left( x,y\right) \right) 
\end{equation}

From Theorem \ref{te3} and hypothesis (iii), it follows that 
\begin{eqnarray}\label{eq60}
&&\left\vert 
\begin{array}{c}
v\left( x,y\right) -\dfrac{\left( \psi \left( y\right) -\psi \left( 0\right) \right) ^{\gamma -1}}{\Gamma \left( \gamma \right) }v\left( x,0\right) -\dfrac{\left( \psi \left( x\right) -\psi \left( 0\right) \right) ^{\gamma -1} }{\Gamma \left( \gamma \right) }v\left( 0,y\right) \\ 
+v\left( 0,0\right) -I_{\theta }^{\alpha ;\psi }f\left( x,y,v\left( x,y\right) ,v_{1}\left( x,y\right) ,v_{2}\left( x,y\right) \right)
\end{array}
\right\vert  \notag \\
&\leq &I_{\theta }^{\alpha ;\psi }\varphi \left( x,y\right) \leq \lambda _{\varphi }^{1}\varphi \left( x,y\right) ,\text{ } x,y\in \left[ 0,\infty \right) 
\end{eqnarray}%
\begin{eqnarray}\label{eq61}
&&\left\vert v_{1}\left( x,y\right) -\frac{\left( \psi \left( y\right) -\psi \left( 0\right) \right) ^{\gamma -1}}{\Gamma \left( \gamma \right) }v_{1}
\left( x,0\right) -I_{0+,y}^{\alpha _{2};\psi }f\left( x,y,v\left(
x,y\right) ,v_{1}\left( x,y\right) ,v_{2}\left( x,y\right) \right) \right\vert  \notag \\
&\leq &I_{0+,y}^{\alpha _{2};\psi }\varphi \left( x,y\right) \leq \lambda _{\varphi }^{2}\varphi \left( x,y\right) ,\text{ } x,y\in \left[ 0,\infty \right) 
\end{eqnarray}%
\begin{eqnarray}\label{eq62}
&&\left\vert v_{2}\left( x,y\right) -\frac{\left( \psi \left( x\right) -\psi \left( 0\right) \right) ^{\gamma -1}}{\Gamma \left( \gamma \right) } v_{2}\left( 0,y\right) -
I_{0+,x}^{\alpha _{1};\psi }f\left( x,y,v\left(
x,y\right) ,v_{1}\left( x,y\right) ,v_{2}\left( x,y\right) \right) \right\vert  \notag \\
&\leq &I_{0+,x}^{\alpha _{1};\psi }\varphi \left( x,y\right) \leq \lambda _{\varphi }^{3}\varphi \left( x,y\right) ,\text{ } x,y\in \left[ 0,\infty \right).
\end{eqnarray}

Using \rm {Eq.(\ref{eq57})}, \rm {Eq.(\ref{eq58})}, \rm {Eq.(\ref{eq59})}, \rm {Eq.(\ref{eq60})}, \rm {Eq.(\ref{eq61})} and \rm {Eq.(\ref{eq62})}, we get 
\begin{eqnarray}
&&\left\vert v\left( x,y\right) -u\left( x,y\right) \right\vert 
\label{eq63} \\
&\leq &\left\vert 
\begin{array}{c}
v\left( x,y\right) -\dfrac{\left( \psi \left( y\right) -\psi \left( 0\right)\right) ^{\gamma -1}}{\Gamma \left( \gamma \right) }v\left( x,0\right) -\dfrac{\left( \psi 
\left( x\right) -\psi \left( 0\right) \right) ^{\gamma -1}}{\Gamma \left( \gamma \right) }v\left( 0,y\right)\notag  \\ \notag+v\left( 0,0\right) -
I_{\theta }^{\alpha ;\psi }f\left( x,y,v\left(x,y\right) ,v_{1}\left( x,y\right) ,v_{2}\left( x,y\right) \right) 
\end{array}
\right\vert   \notag \\
&&+I_{\theta }^{\alpha ;\psi }\left\vert f\left( x,y,v\left( x,y\right),v_{1}\left( x,y\right) ,v_{2}\left( x,y\right) \right) -f\left( x,y,u\left(x,y\right) ,
u_{1}\left( x,y\right) ,u_{2}\left( x,y\right) \right)\right\vert   \notag \\
&\leq &\lambda _{\varphi }^{1}\varphi \left( x,y\right) +I_{\theta }^{\alpha;\psi }\left\{ L_{f}\left( x,y\right) \underset{i\in \left\{ 1,2,3\right\} }{
\max }\left\vert v_{i}\left( x,y\right) -u_{i}\left( x,y\right) \right\vert
\right\} .
\end{eqnarray}

From Lemma \ref{l1}, it follows that
\begin{eqnarray*}
&&\left\vert v\left( x,y\right) -u\left( x,y\right) \right\vert   \notag
\label{eq64} \\
&\leq &\lambda _{\varphi }^{1}\varphi \left( x,y\right) \mathbb{E}_{\alpha } \left[ L_{f}\left( x,y\right) \Gamma \left( \alpha _{1}\right) \Gamma \left( \alpha _{2}\right) \left( \psi \left( b\right) -\psi \left( t\right) \right)^{\alpha _{2}}\left( \psi \left( a\right) -\psi \left( s\right) \right) ^{\alpha _{1}}\right]   \notag \\
&\leq &\lambda _{\varphi }^{1}\varphi \left( x,y\right) \mathbb{E}_{\alpha } \left[ L_{f}\left( x,y\right) \Gamma \left( \alpha _{1}\right) \Gamma \left( \alpha _{2}\right) \left( \psi \left( \infty \right) -\psi \left( 0\right)
\right) ^{\alpha _{2}}\left( \psi \left( \infty \right) -\psi \left( 0\right) \right) ^{\alpha _{1}}\right]   \notag \\ 
&=&C_{f,\varphi }^{1}\varphi \left( x,y\right) 
\end{eqnarray*}
where
\begin{equation*}
C_{f,\varphi }^{1}:=\lambda _{\varphi }^{1}\mathbb{E}_{\alpha }\left[ L_{f}\left( x,y\right) \Gamma \left( \alpha _{1}\right) \Gamma \left( \alpha _{2}\right) \left( \psi \left( \infty \right) -\psi \left( 0\right) \right) ^{\alpha _{2}}\left( \psi \left( \infty \right) -\psi \left( 0\right) \right) ^{\alpha _{1}}\right] 
\end{equation*}
$ x,y\in \left[ 0,\infty \right)$, $\dfrac{1}{2}<\alpha \leq 1$ and $ \psi \left( \infty \right) <\infty$.

Similarly, we have 
\begin{eqnarray*}
\left\vert v_{1}\left( x,y\right) -u_{1}\left( x,y\right) \right\vert  &\leq &\lambda _{\varphi }^{2}\varphi \left( x,y\right) \mathbb{E}_{\alpha }\left[ L_{f} \left( x,y\right) \Gamma \left( \alpha _{1}\right) \left( \psi \left( \infty \right) -\psi \left( 0\right) \right) ^{\alpha _{1}}\right]   \notag \\
&=&C_{f,\varphi }^{2}\varphi \left( x,y\right) 
\end{eqnarray*}
where
\begin{equation*}
C_{f,\varphi }^{2}:=\lambda _{\varphi }^{2}\mathbb{E}_{\alpha }\left[ L_{f}\left( x,y\right) \Gamma \left( \alpha _{1}\right) \left( \psi \left( \infty \right) -\psi \left( 0\right) \right) ^{\alpha _{1}}\right] 
\end{equation*}
$ x,y\in \left[ 0,\infty \right)$, $\dfrac{1}{2}<\alpha \leq 1$ and $\psi \left( \infty \right) <\infty $.

Also, 
\begin{eqnarray*}
\left\vert v_{2}\left( x,y\right) -u_{2}\left( x,y\right) \right\vert  &\leq &\lambda _{\varphi }^{3}\varphi \left( x,y\right) \mathbb{E}_{\alpha } \left[ L_{f}\left( x,y\right) \Gamma \left( \alpha _{2}\right) \left( \psi \left( \infty \right) -\psi \left( 0\right) \right) ^{\alpha _{2}}\right]   \notag \\
&=&C_{f,\varphi }^{3}\varphi \left( x,y\right) 
\end{eqnarray*}
where
\begin{equation*}
C_{f,\varphi }^{3}:=\lambda _{\varphi }^{3}\mathbb{E}_{\alpha }\left[ L_{f}\left( x,y\right) \Gamma \left( \alpha _{2}\right) \left( \psi \left( \infty \right) -\psi \left( 0\right) \right) ^{\alpha _{2}}\right] 
\end{equation*}
$ x,y\in \left[ 0,\infty \right)$, $\dfrac{1}{2}<\alpha \leq 1$ and $\psi \left( \infty \right) <\infty $.

Therefore, the solution of Eq.(\ref{eq3}) is generalized Ulam-Hyers-Rassias stable.

\end{proof}


\section{Concluding remarks}

By means of the $\psi$-Riemann-Liouville fractional partial integral and the $\psi$-Hilfer fractional partial derivative of $N$ variables, we investigated a fractional partial differential equation of hyperbolic
type and showed that its solutions admits Ulam-Hyers and Ulam-Hyers-Rassias stabilities in a Banach space $\mathbb{B}$. Besides, we have also made some remarks which we consider extremely important, in particular the ones concerning the wide class of fractional partial integrals and derivatives obtained from these new
extensions, as in Remarks 1 and 2. 

It is important to note that the extension of the $\psi$-Hilfer fractional derivative in one variable to the case of $N$ variables shows that the use of fractional calculus can be naturally extended in the case of models involving partial differential equations, in the study of the existence and uniqueness of solutions of such PDEs and in the analytical continuation of solutions of partial differential equations.  In a forthcoming work, we study the Volterra integral equations and various types of stability of partial differential equations and other properties for $\psi$-Hilfer fractional partial derivative.

\section*{Acknowledgment}
The authors are grateful to Dr. J. Em\'{\i}lio Maiorino for several useful discussions.

\bibliography{ref}
\bibliographystyle{plain}

\end{document}